\documentclass[11pt,reqno,a4]{amsart}

\let\Setlength\setlength 
\usepackage{calc}
\newlength{\arrayrulewidthOriginal}
\newcommand{\Cline}[2]{%
  \noalign{\global\Setlength{\arrayrulewidthOriginal}{\arrayrulewidth}}%
  \noalign{\global\Setlength{\arrayrulewidth}{#1}}\cline{#2}%
  \noalign{\global\Setlength{\arrayrulewidth}{\arrayrulewidthOriginal}}}

\usepackage[english]{babel}
\usepackage{amsthm}
\usepackage{graphicx}
\usepackage{xcolor} 
\usepackage{transparent}
\usepackage{amssymb}
\usepackage{amsmath}
\usepackage{mathtools}
\usepackage[arrow, matrix, curve]{xy}
\usepackage{units} 
\usepackage{MnSymbol} 
\usepackage{here} 
\usepackage{enumerate} 
\usepackage{comment} 
\usepackage{listings} 
\usepackage{pdfpages} 
\usepackage{multirow} 
\usepackage{bigdelim}
\usepackage{multicol} 
\usepackage{cancel} 
\usepackage{wrapfig} 
\usepackage{float}
\usepackage{hyperref}
\usepackage{tikz-cd}
\usepackage{enumitem} 
\usepackage{bm}
\usepackage{colortbl} 
\usepackage{bm} 
\usepackage{dsfont}

\usetikzlibrary{%
  matrix,%
  calc,%
  arrows%
}

\restylefloat{table}

\hypersetup{
    colorlinks,
    linkcolor={red!70!black},
    citecolor={green!60!black},
    urlcolor={blue!80!black}
}

\theoremstyle{dotless}
\newtheorem{corollary}{Corollary}[section]
\newtheorem{lemma}[corollary]{Lemma}
\newtheorem{theorem}[corollary]{Theorem}
\newtheorem{proposition}[corollary]{Proposition}
\newtheorem*{theorem*}{Theorem}
\theoremstyle{definition}
\newtheorem{construction}[corollary]{Construction}
\newtheorem{definition}[corollary]{Definition}
\newtheorem{remark}[corollary]{Remark}
\newtheorem{example}[corollary]{Example}

\newtheorem{classification}[corollary]{Classification}
\newtheorem{notation}[corollary]{Notation}

\newtheorem{observation}[corollary]{Observation}
\newtheorem{open problem}[corollary]{Open problem}

\newcommand{\CR}{\operatorname{cr}}

\newcommand{\mult}{\operatorname{mult}}

\newcommand{\ft}{\operatorname{ft}}
\newcommand{\ev}{\operatorname{ev}}

\newcommand{\val}{\operatorname{val}}

\newcommand{\sgn}{\operatorname{sgn}}

\begin{document}
\pagenumbering{arabic}

\author{Christoph Goldner}
\address{ Eberhard Karls Universit\"{a}t T\"{u}bingen, Germany}
\email{\href{mailto:christoph.goldner@math.uni-tuebingen.de}{christoph.goldner@math.uni-tuebingen.de}}

\title{Generalizing tropical Kontsevich's formula to multiple cross-ratios}


\begin{abstract}
\textit{Kontsevich's formula} is a recursion that calculates the number of rational degree $d$ curves in $\mathbb{P}_{\mathbb{C}}^2$ passing through $3d-1$ general positioned points. Kontsevich proved it by considering curves that satisfy extra conditions besides the given point conditions. These crucial extra conditions are two line conditions and a condition called \textit{cross-ratio}.

This paper addresses the question whether there is a \textit{general Kontsevich's formula} which holds for more than one cross-ratio.
Using tropical geometry, we obtain such a recursive formula. For that we use a correspondence theorem \cite{IlyaCRC} that relates the algebro-geometric numbers in question to tropical ones.
It turns out that the general tropical Kontsevich's formula we obtain is capable of not only computing the algebro-geometric numbers we are looking for, but also of computing further tropical numbers for which there is no correspondence theorem yet.

We show that our recursive general Kontsevich's formula implies the original Kontsevich's formula and that the initial values are the numbers Kontsevich's fomula provides and purely combinatorial numbers, so-called \textit{cross-ratio multiplicities}.
\end{abstract}

\maketitle


\section*{Introduction}
Consider the following enumerative problem: Determine the number $N_d$ of rational degree $d$ curves in $\mathbb{P}_{\mathbb{C}}^2$ passing through $3d-1$ general positioned points. For small $d$, this question can be answered using methods from classical algebraic geometry. It took until '94 when Kontsevich, inspired from developments in physics, presented a recursive formula to calculate the numbers $N_d$ for all degrees.
\begin{theorem*}[Kontsevich's formula, \cite{KontsevichOriginal}]
The numbers $N_d$ are determined by the recursion

\begin{align*}
N_{d}=
\sum_{\substack{d_1+d_2=d \\ d_1,d_2>0}}
\left( d_1^2 d_2^2 \cdot {3d-4 \choose 3d_1-2}- d_1^ 3d_2 \cdot {3d-4 \choose 3d_1-1} \right) N_{d_1}N_{d_2}
\end{align*}
with initial value $N_1=1$.
\end{theorem*}

\noindent This recursion is known as \textit{Kontsevich's formula}. The only initial value it needs is $N_1=1$, i.e. the fact that there is exactly one line passing through two different points.

A \textit{cross-ratio} is an element of the ground field associated to four collinear points. It encodes the relative position of these four points to each other. It is invariant under projective transformations and can therefore be used as a constraint that four points on $\mathbb{P}^1$ should satisfy. So a cross-ratio can be viewed as a condition on elements of the moduli space of $n$-pointed rational stable maps to a toric variety.

A crucial idea in the proof of Kontsevich's formula is to consider curves that satisfy extra conditions besides the given point conditions. These extra conditions are two line conditions and a cross-ratio condition.
In fact, the original proof of Kontsevich's formula yields a formula to determine the number of rational plane curves satisfying an appropriate number of general positioned point conditions, two line conditions and one cross-ratio condition.
Hence the following question naturally comes up:

\fbox{\parbox{15cm}{Is there a general version of Kontsevich's formula that recursively calculates the number of rational plane degree $d$ curves that satisfy general positioned point, curve and cross-ratio conditions?}}\\

\noindent We remark that Kontsevich's formula was generalized in different ways before, e.g. Ernstr\"{o}m and Kennedy took tangency conditions into account \cite{ErnstroemKennedy98,ErnstroemKennedy99} and Di Francesco and Itzykson \cite{FrancescoItzykson} generalized it among others to $\mathbb{P}_{\mathbb{C}}^1\times \mathbb{P}_{\mathbb{C}}^1$. We are not aware of any generalization that includes multiple cross-ratios.\\

Tropical geometry proved to be an effective tool to answer enumerative questions. To successfully apply tropical geometry to an enumerative problem, a so-called \textit{correspondence theorem} is required.
The first celebrated correspondence theorem was proved by Mikhalkin \cite{MikhalkinFundamental}. It states that the numbers $N_d$ equal its tropical counterpart, i.e. they can be obtained from the weighted\footnote{Tropical curves are always counted with multiplicity.} count of rational tropical degree $d$ curves in $\mathbb{R}^2$ passing through $3d-1$ general positioned points. Hence Kontsevich's formula translates into a recursion on the tropical side called \textit{tropical Kontsevich's formula} and vice versa.
Gathmann and Markwig demonstrated the efficiency of tropical methods by giving a purely tropical proof of tropical Kontsevich's formula \cite{KontsevichPaper}. Applying Mikhalkin's correspondence theorem then yields Kontsevich's formula.

In the tropical proof --- as in the non-tropical case --- rational tropical degree $d$ curves that satisfy point conditions, two line conditions and one \textit{tropical cross-ratio} condition are considered.
Roughly speaking, a tropical cross-ratio fixes the sum of lengths of a collection of bounded edges of a rational tropical curve.

\begin{example}\label{ex:introduction_trop_CR}
Figure \ref{Figure_25} shows a plane rational tropical degree $2$ curve $C$ such that $C$ satisfies four point conditions with its contracted ends labeled by $1,2,4,5$, and such that $C$ satisfies one curve condition (which is a line that is indicated by dots) with its contracted end labeled with $3$. Moreover, $C$ satisfies the tropical cross-ratio $\lambda'=(12|34)$ which determines the bold red length.

\begin{figure}[H]
\centering
\def\svgwidth{450pt}
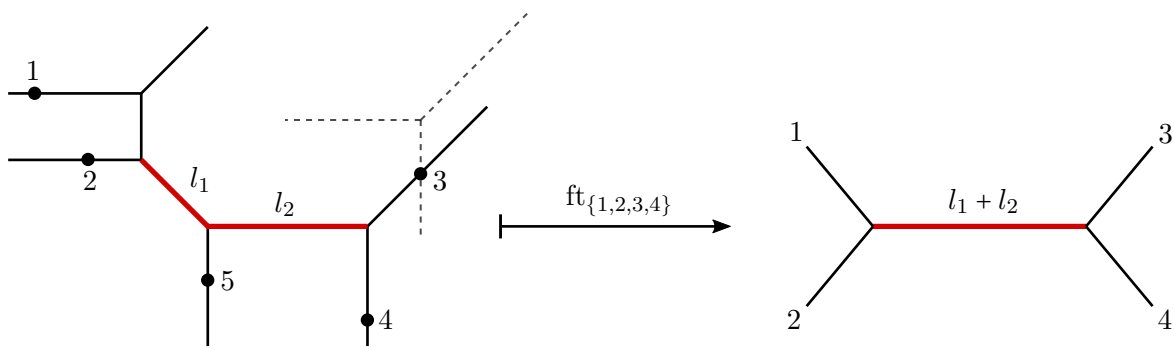
\caption{On the left there is the curve $C$ of Example \ref{ex:introduction_trop_CR} with its bounded edges that contribute to the tropical cross-ratio $\lambda'$ colored bold red (the lengths of these edges are $l_1,l_2$). On the right there is the image of $C$ under a so-called \textit{forgetful map} $\ft_{\lbrace 1,2,3,4,\rbrace}$ that records the labels and the length $l_1+l_2$ which appear in the tropical cross-ratio $\lambda'$.}
\label{Figure_25}
\end{figure}
\end{example}

Tropical cross-ratios are the tropical counterpart to non-tropical cross-ratios.
Mikhalkin \cite{MikhalkinCRC} introduced a tropical version of cross-ratios under the name ``tropical double ratio'' to embed the moduli space of $n$-marked abstract rational tropical curves $\mathcal{M}_{0,n}$ into $\mathbb{R}^N$ in order to give it the structure of a balanced fan.
Tyomkin proved a correspondence theorem \cite{IlyaCRC} that involves cross-ratios, where the length of a tropical cross-ratio is related to a given non-tropical cross-ratio via the valuation map. More precisely, Tyomkin's correspondence theorem states that the number of rational plane degree $d$ curves satisfying point and cross-ratio conditions equals its tropical counterpart. Hence a general tropical Kontsevich's formula that recursively computes the weighted number of rational plane tropical curves of degree $d$ that satisfy point and tropical cross-ratio conditions simultaneously computes the non-tropical numbers as well.\\

Our approach to a general Kontsevich's formula is inspired by the one of Gathmann and Markwig.
Let us sum up the (for our purposes) most relevant ideas and techniques used in \cite{KontsevichPaper}:
\begin{description}
\item[1 Splitting curves]\hfill \\
An important observation is that a count of tropical curves satisfying a tropical cross-ratio condition $\lambda'$ is independent of the length of the tropical cross-ratio. In particular, one can choose a large length for $\lambda'$. An even more important observation, which, at the end of the day, gives rise to a recursion is the following: If the length of $\lambda'$ is large enough, then all tropical curves satisfying $\lambda'$ have a contracted bounded edge. Hence they can be split into two curves.
\item[2 Splitting multiplicities]\hfill \\
Tropical curves are counted with multiplicities. So splitting curves using a large length for a tropical cross-ratio only yields a recursion if the multiplicities of such tropical curves split accordingly.
\item[3 Using rational equivalence]\hfill \\
A tropical cross-ratio appears as a pull-back of a point of $\mathcal{M}_{0,4}$ and pull-backs of different point of $\mathcal{M}_{0,4}$ are rationally equivalent \cite{FirstStepsIntersectionTheory}. Hence the number of tropical curves satisfying a tropical cross-ratio $\lambda'=(\beta_1\beta_2|\beta_3\beta_4)$ does not depend on how the labels $\beta_1,\dots,\beta_4$ are grouped together --- we could also consider the cross-ratio $\tilde{\lambda}'=(\beta_1\beta_3|\beta_2\beta_4)$ and obtain the same number. This yields an equation.
\end{description}

As a result, we obtain a general tropical Kontsevich's formula (Theorem \ref{thm:generalized_kontsevich}) that recursively calculates the weighted number of rational plane tropical curves of degree $d$ that satisfy point conditions, curve conditions and tropical cross-ratio conditions. In order to obtain a non-tropical general Kontsevich's formula (Corollary \ref{cor:non-tropical_general_Kontsevich_formula}), we apply Tyomkin's correspondence theorem \cite{IlyaCRC}. Notice that Tyomkin's correspondence theorem only holds for point and cross-ratio conditions. There is no correspondence theorem that relates the tropical numbers that also involve curve conditions to their non-tropical counterparts yet.

The general Kontsevich's formula we derive this way allows us to recover Kontsevich's fomula, see Corollary \ref{cor:tropical_kontsevich_formula}. The initial values of the general Kontsevich's formula are the numbers provided by the original Kontsevich's formula and so-called \textit{cross-ratio multiplicities}, which are purely combinatorial \cite{CR1}.

\subsection*{Organization of the paper}
We use the framework provided by steps 1 to 3 described above to obtain a general Kontsevich's formula. Although this general framework follows the outline of the tropical proof of Kontsevich's formula in \cite{KontsevichPaper}, new methods for steps 1 and 2 are required, which we elaborate right after the preliminary section. The preliminary section collects background on tropical moduli spaces and tropical intersection theory. For step 1, a new and general concept of moving parts of a tropical curve is established. Splitting the multiplicities in the 2nd step is done via a novel approach that considers ``artificial'' line conditions.
Putting everything together to deduce our recursion is done in the last section. To complete the paper, we conclude (tropical) Kontsevich's formula from our general version.

\subsection*{Acknowledgements}
The author would like to thank Hannah Markwig for valuable feedback and helpful discussions. The author gratefully acknowledges partial support by DFG-collaborative research center TRR 195 (INST 248/237-1).
This work was partially completed during the workshop ``Tropical Geometry: new directions'' at the Mathematisches Forschungsinstitut Oberwolfach in spring 2019. The author would like to thank the institute for hospitality and excellent working conditions.

\section{Preliminaries}
We recall some standard notations and definitions from tropical geometry \cite{MikhalkinCRC, KontsevichPaper,GathmannKerberMarkwig} and give a very brief overview of the necessary tropical intersection theory. After that, (degenerated) tropical cross-ratios are defined \cite{CR1}.

Besides this, we try to make notations used as clear as possible by introducing notations in separate blocks to which we refer later.

\begin{notation}\label{notation:underlined_symbols}
We write $[m]:=\lbrace 1,\dots,m \rbrace$ if $0\neq m\in\mathbb{N}$, and if $m=0$, then define $[m]:=\emptyset$.
Underlined symbols indicate a set of symbols, e.g. $\underline{n}\subset [m]$ is a subset $\lbrace 1,\dots, m\rbrace$. We may also use sets $S$ of symbols as an index, e.g. $p_S$, to refer to the set of all symbols $p$ with indices taken from $S$, i.e. $p_S:=\lbrace p_i\mid i\in S \rbrace$. The $\#$-symbol is used to indicate the number of elements in a set, for example $\#[m]=m$.
\end{notation}

\subsection*{Tropical moduli spaces}
This subsection collects background from \cite{MikhalkinCRC, KontsevichPaper,GathmannKerberMarkwig}.

\begin{definition}[Moduli space of abstract rational tropical curves]
We use Notation \ref{notation:underlined_symbols}. An \textit{abstract rational tropical curve} is a metric tree $\Gamma$ with unbounded edges called \textit{ends} and with $\val(v)\geq 3$ for all vertices $v\in\Gamma$. It is called $N$-\textit{marked abstract tropical curve} $(\Gamma,x_{[N]})$ if $\Gamma$ has exactly $N$ ends that are labeled with pairwise different $x_1,\dots,x_N\in\mathbb{N}$. Two $N$-marked tropical curves $(\Gamma,x_{[N]})$ and $(\tilde{\Gamma},\tilde{x}_{[N]})$ are isomorphic if there is a homeomorphism $\Gamma\to \tilde{\Gamma}$ mapping $x_i$ to $\tilde{x}_i$ for all $i$ and each edge of $\Gamma$ is mapped onto an edge of $\tilde{\Gamma}$ by an affine linear map of slope $\pm 1$. The set $\mathcal{M}_{0,N}$ of all $N$-marked tropical curves up to isomorphism is called \textit{moduli space of $N$-marked abstract tropical curves}. Forgetting all lengths of an $N$-marked tropical curve gives us its \textit{combinatorial type}.
\end{definition}

\begin{remark}[$\mathcal{M}_{0,N}$ is a tropical fan]
The moduli space $\mathcal{M}_{0,N}$ can explicitly be embedded into a $\mathbb{R}^t$ such that $\mathcal{M}_{0,N}$ is a tropical fan of pure dimension $N-3$ with its fan structure given by combinatorial types and all its weights are one, i.e. $\mathcal{M}_{0,n}$ represents an affine cycle in $\mathbb{R}^t$. This allows us to use tropical intersection theory on $\mathcal{M}_{0,n}$.
\end{remark}

\begin{figure}[H]
\centering
\def\svgwidth{150pt}
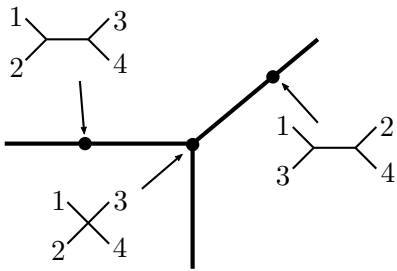
\caption{One way of embedding the moduli space $\mathcal{M}_{0,4}$ into $\mathbb{R}^2$ centered at the origin of $\mathbb{R}^2$. The length of a bounded edge of an abstract tropical curve depicted above is given by the distance of the point in $\mathcal{M}_{0,4}$ corresponding to this curve from the origin of $\mathbb{R}^2$. The ends of $\mathcal{M}_{0,4}$ correspond to different distributions of labels on ends of abstract tropical curves with four ends. All cases are $(12|34), (13|24), (14|23)$.}
\label{Example_M_0_4}
\end{figure}

\begin{definition}[Moduli space of rational tropical stable maps to $\mathbb{R}^2$]\label{definition:moduli_stable_maps}
Let $m,d\in\mathbb{N}$. A \textit{rational tropical stable map of degree $d$ to $\mathbb{R}^2$ with $m$ contracted ends} is a tuple $(\Gamma,x_{[N]},h)$ with $N\in\mathbb{N}_{>0}$, where $(\Gamma,x_{[N]})$ is an $N$-marked abstract tropical curve with $N=3d+m$, $x_{[N]}=[N]$ and a map $h:\Gamma\to\mathbb{R}^2$ that satisfies the following:
\begin{itemize}
\item[(a)]
Let $e\in\Gamma$ be an edge with length $l(e)\in [0,\infty]$, identify $e$ with $[0,l(e)]$ and denote the vertex of $e$ that is identified with $0\in [0,l(e)]=e$ by $V$. The map $h$ is integer affine linear, i.e. $h\mid_e:t\mapsto tv+a$ with $a\in\mathbb{R}^2$ and $v(e,V):=v\in\mathbb{Z}^2$, where $v(e,V)$ is called \textit{direction vector of $e$ at $V$} and the \textit{weight} of an edge (denoted by $\omega(e)$) is the $\gcd$ of the entries of $v(e,V)$. The vector $\frac{1}{\omega(e)}\cdot v(e,V)$ is called the \textit{primitive} direction vector of $e$ at $V$. If $e=x_i\in\Gamma$ is an end, then $v(x_i)$ denotes the direction vector of $x_i$ pointing away from its one vertex it is adjacent to.
\item[(b)]
The direction vector $v(x_i)$ of an end labeled with $x_i$ is given by 
\begin{align*}
\begin{array}{c||c|c|c|c}
x_i & 0,\dots,m & m+1,\dots, m+d & m+d+1,\dots, m+2d & m+2d+1,\dots, m+3d \\
\hline\hline
v(x_i) & \begin{pmatrix}
0\\
0
\end{pmatrix}
&
\begin{pmatrix}
-1\\
0
\end{pmatrix}
&
\begin{pmatrix}
0\\
-1
\end{pmatrix}
&
\begin{pmatrix}
1\\
1
\end{pmatrix}
\end{array}.
\end{align*}
Ends with direction vector zero are called \textit{contracted ends}.
\item[(c)]
The \textit{balancing condition}
\begin{align*}
\sum_{\substack{e\in\Gamma\textrm{ an edge}, \\ V \textrm{ vertex of }e}}v(e,V)=0
\end{align*}
holds for every vertex $V\in\Gamma$.
\end{itemize}
Two rational tropical stable maps of degree $d$ with $m$ contracted ends, namely $(\Gamma , x_{[N]},h)$ and $(\Gamma' ,x'_{[N]},h')$, are isomorphic if there is an isomorphism $\varphi$ of their underlying $N$-marked tropical curves such that $h'\circ\varphi=h$.
The set $\mathcal{M}_{0,m}\left(\mathbb{R}^2,d\right)$ of all (rational) tropical stable maps of degree $\Delta$ to $\mathbb{R}^2$ with $m$ contracted ends up to isomorphism is called \textit{moduli space of (rational) tropical stable maps of degree $d$ to $\mathbb{R}^2$ (with $m$ contracted ends)}.
\end{definition}

\begin{remark}[$\mathcal{M}_{0,m}\left(\mathbb{R}^2,d\right)$ is a fan]\label{remark:identification:stable_maps_abstract_maps}
The map
\begin{align*}
\mathcal{M}_{0,m}\left(\mathbb{R}^2,d\right) &\to \mathcal{M}_{0,N}\times\mathbb{R}^2\\
(\Gamma,x_{[N]},h) &\mapsto \left(\left(\Gamma,x_{[N]}\right),h(x_1)\right)
\end{align*}
with $N=3d+m$ is bijective and $\mathcal{M}_{0,m}\left(\mathbb{R}^2,d\right)$ is a tropical fan of dimension $3d+m-1$. Hence $\mathcal{M}_{0,n}\left(\mathbb{R}^2,d\right)$ represents an affine cycle in a $\mathbb{R}^t$. This allows us to use tropical intersection theory on $\mathcal{M}_{0,n}\left(\mathbb{R}^2,d\right)$.
\end{remark}

\begin{definition}[Evaluation maps]\label{def:ev_map}
For $i\in [m]$, the map
\begin{align*}
\operatorname{ev}_i:\mathcal{M}_{0,m}\left(\mathbb{R}^2,d\right) &\to \mathbb{R}^2\\
(\Gamma,x_{[N]},h) &\mapsto h(x_i)
\end{align*}
is called \textit{$i$-th evaluation map}. Under the identification from Remark \ref{remark:identification:stable_maps_abstract_maps} the $i$-th evaluation map is a morphism of fans $\operatorname{ev}_i:\mathcal{M}_{0,N}\times\mathbb{R}^2 \to \mathbb{R}^2$. This allows us to pull-back cycles via the evaluation map.
\end{definition}

\begin{definition}[Forgetful maps]\label{def:ft_map}
For $N\geq4$ the map
\begin{align*}
\operatorname{ft}_{x_{[N-1]}}:\mathcal{M}_{0,N}&\to\mathcal{M}_{0,N-1}\\
(\Gamma,x_{[N]}) &\mapsto (\Gamma',x_{[N-1]}),
\end{align*}
where $\Gamma'$ is the stabilization (straighten $2$-valent vertices) of $\Gamma$ after removing its end marked by $x_N$ is called the $N$-th \textit{forgetful map}. Applied recursively, it can be used to forget several ends with markings in $I^C\subset x_{[N]}$, denoted by $\operatorname{ft}_I$, where $I^C$ is the complement of $I\subset x_{[N]}$. With the identification from Remark \ref{remark:identification:stable_maps_abstract_maps}, and additionally forgetting the map $h$ to the plane, we can also consider 
\begin{align*}
\operatorname{ft}_I:\mathcal{M}_{0,m}\left(\mathbb{R}^2,d\right) &\to\mathcal{M}_{0,|I|}\\
(\Gamma,x_{[N]},h) &\mapsto \operatorname{ft}_I(\Gamma,x_i|i\in I).
\end{align*}
\end{definition}
Any forgetful map is a morphism of fans. This allows us to pull-back cycles via the forgetful map.

\begin{definition}[Tropical curves and multi lines]\label{def:tropical_line}
A \textit{plane tropical curve $C$ of degree $d$} is the abstract $1$-dimensional cycle a rational tropical stable map of degree $d$ gives rise to, i.e. $C$ is an embedded $1$-dimensional polyhedral complex in $\mathbb{R}^2$. A \textit{(tropical) multi line} $L$ is a tropical rational curve in $\mathbb{R}^2$ with $3$ ends such that the primitive direction of each of this ends is one of the \textit{standard directions} $(-1,0),(0,-1)$ or $(1,1)\in \mathbb{R}^2$. The weight with which an end of $L$ appears is denoted by $\omega(L)$.
\end{definition}

\subsection*{Tropical intersection products}
As indicated in the last section, tropical intersection theory can be applied to the tropical moduli spaces that are interesting for us. For a short and --- for our purposes --- sufficient introduction to tropical intersection theory have a look at the preliminary section of \cite{CR1}. For more background of tropical intersection theory see \cite{FultonSturmfels, Rau,Allermann,FirstStepsIntersectionTheory, Katz2012, IntersectionMatroidalFans,  AllermannHampeRau, JohannesIntersectionsonTropModuliSpaces}. In the present paper tropical intersection theory provides the overall framework in which we work but all we need from this machinary is the following:

\begin{remark}[Enumerative meaning of our tropical intersection products]\label{remark:enumerative_meaning_intersection_product}
Throughout this paper, we consider intersection products of the form $\varphi_1^*(Z_1)\cdots\varphi_r^*(Z_r)\cdot \mathcal{M}_{0,m}\left(\mathbb{R}^2,d\right)$, where $\varphi_i$ is either an evaluation $\ev_i$ map from Definition \ref{def:ev_map} or a forgetful map $\ft_I$ to $\mathcal{M}_{0,4}$ from Definition \ref{def:ft_map}, and $Z_i$ is a cycle we want to pull-back via $\varphi_i$ for $i\in [r]$.
Notice that $\ev_i$ is a map to $\mathbb{R}^2$ while $\ft_I$ is a map to $\mathcal{M}_{0,4}$.
Using a projection $\tilde{\pi}:\mathcal{M}_{0,4}\to\mathbb{R}$ as in Remark 2.2 of \cite{CR1} and considering $\tilde{\pi}\circ\ft_I$ instead of $\ft_I$ does not affect $\varphi_1^*(Z_1)\cdots\varphi_r^*(Z_r)\cdot \mathcal{M}_{0,m}\left(\mathbb{R}^2,d\right)$ since
\begin{align*}
\left( \tilde{\pi}\circ\ft_I\right)^*(\tilde{Z}_i)&=\ft_I^*\left( \tilde{\pi}^*(\tilde{Z}_i)\right)\\
&=\ft_I^*(Z_i)
\end{align*}
holds for a suitable cycle $\tilde{Z}_i$. Thus all our maps can be treated as maps to either $\mathbb{R}^2$ or $\mathbb{R}^1$. Hence Proposition 1.15 of \cite{JohannesIntersectionsonTropModuliSpaces} can be applied, and together with Proposition 1.12 of \cite{JohannesIntersectionsonTropModuliSpaces} and Lemma 2.11 of \cite{CR1} it follows that the support of the intersection product $\varphi_1^*(Z_1)\cdots\varphi_r^*(Z_r)\cdot \mathcal{M}_{0,m}\left(\mathbb{R}^2,d\right)$ equals $\varphi_1^{-1}(Z_1)\cap\dots \cap \varphi_r^{-1}(Z_r)$. Hence this intersection product gains an enumerative meaning if it is $0$-dimensional. More precisely, each point in such an intersection product corresponds to a tropical stable map that satisfies certain conditions that are given by the cycles $Z_i$ for $i\in [r]$.
\end{remark}

The weights of such intersection products $\varphi_1^*(Z_1)\cdots\varphi_r^*(Z_r)\cdot \mathcal{M}_{0,m}\left(\mathbb{R}^2,d\right)$ are discussed within the next section. Before proceeding with the next section, we want to briefly recall the concept of rational equivalence that is then frequently used in this paper.

\begin{remark}[Rational equivalence]\label{remark:facts_about_rational_equivalence}
When considering cycles $Z_i$ as in Remark \ref{remark:enumerative_meaning_intersection_product} that are conditions we impose on tropical stable maps, then we usually want to ensure that a $0$-dimensional cycle $\varphi_1^*(Z_1)\cdots\varphi_r^*(Z_r)\cdot \mathcal{M}_{0,m}\left(\mathbb{R}^2,d\right)$ is independent of the exact positions of the conditions $Z_i$ for $i\in [r]$. This is where \textit{rational equivalence} comes into play. We usually consider cycles like $Z_i$ up to a rational equivalence relation. 
The most important facts about this relation are the following:
\begin{itemize}
\item[(a)]
Two cycles $Z,Z'$ in $\mathbb{R}^n$ that only differ by a translation are rationally equivalent.
\item[(b)]
Pull-backs $\varphi^*(Z),\varphi^*(Z')$ of rationally equivalent cycles $Z,Z'$ are rationally equivalent.
\item[(c)]
The \textit{degree} of a $0$-dimensional intersection product which is defined as the sum of all weights of all points in this intersection product is compatible with rational equivalence, i.e. if two $0$-dimensional intersection products are rationally equivalent, then their degrees are the same.
\end{itemize}
Notice that (a)-(c) allows us to ``move" all conditions we consider slightly without affecting a count of tropical stable maps we are interested in.
\end{remark}

Another fact about rational equivalence is the following:

\begin{remark}[Recession fan]\label{remark:recession_fan}
Each tropical curve $C$ of degree $d$ in $\mathbb{R}^2$ is rationally equivalent to a multi line $L_C$ with weights $\omega(L_C)=d$. Hence pull-backs of $C$ and $L_C$ along the evaluation maps are rationally equivalent. The multi line $L_C$ is also called \textit{recession fan of $C$}.
\end{remark}

\subsection*{Tropical cross-ratios and the numbers of interest}

\begin{definition}
A \textit{(tropical) cross-ratio} $\lambda'$ is an unordered pair of pairs of unordered numbers $\left(\beta_1\beta_2|\beta_3\beta_4\right)$ together with an element in $\mathbb{R}_{>0}$ denoted by $|\lambda'|$, where $\beta_1,\dots,\beta_4$ are labels of pairwise distinct ends of a tropical stable map of $\mathcal{M}_{0,m}\left(\mathbb{R}^2,d \right)$. We say that $C\in\mathcal{M}_{0,m}\left(\mathbb{R}^2,d \right)$ satisfies the cross-ratio constraint $\lambda'$ if $C\in\ft^*_{\lambda'}\left(|\lambda'| \right)\cdot \mathcal{M}_{0,m}\left(\mathbb{R}^2,d \right)$, where $|\lambda'|$ is the canonical local coordinate of the ray $\left(\beta_1\beta_2|\beta_3\beta_4\right)$ in $\mathcal{M}_{0,4}$.
Figure \ref{Figure_25} of Example \ref{ex:introduction_trop_CR} in the introduction provides an example of a tropical stable map satisfying a non-degenerated cross-ratio $\lambda'$ with length $|\lambda'|=l_1+l_2$.

A \textit{degenerated (tropical) cross-ratio} $\lambda$ is defined as a set $\lbrace \beta_1,\dots,\beta_4\rbrace$, where $\beta_1,\dots,\beta_4$ are pairwise distinct labels of ends of a tropical stable map $\mathcal{M}_{0,m}\left(\mathbb{R}^2,d \right)$. We say that $C\in\mathcal{M}_{0,m}\left(\mathbb{R}^2,d \right)$ satisfies the degenerated cross-ratio constraint $\lambda$ if $C\in\ft^*_\lambda\left(0 \right)\cdot \mathcal{M}_{0,m}\left(\mathbb{R}^2,d \right)$. A degenerated cross-ratio arises from a non-degenerated cross-ratio by taking $|\lambda'|\to 0$ (see \cite{CR1} for more details). We refer to $\lambda$ as \textit{degeneration} of $\lambda'$ in this case.

Throughout the paper, we stick to the convention to denote a non-degenerated cross-ratio by $\lambda'$ and a degenerated one by $\lambda$.
\end{definition}

\begin{definition}\label{def:cycle_Z_d}
Let $m\in\mathbb{N}_{>0}$. Let $\lbrace\underline{n},\underline{\kappa},\underline{f}\rbrace$ be a partition of the set $[m]$, i.e. $\underline{n},\underline{\kappa},\underline{f}\subset [m]$ and $\underline{n}\cupdot \underline{\kappa}\cupdot \underline{f}=[m]$.
Consider a degree $d\in\mathbb{N}$, $\tilde{l}\in\mathbb{N}$ degenerated cross-ratios $\lambda_{[\tilde{l}]}$, $l'\in\mathbb{N}$ non-degenerated cross-ratios $\mu'_{[l']}$, points $p_{\underline{n}}\in\mathbb{R}^2$ and tropical multi lines $L_{\underline{\kappa}}$. Define the cycle
\begin{align*}
Z_{d}\left(p_{\underline{n}},L_{\underline{\kappa}}, \lambda_{[\tilde{l}]},\mu'_{[l']} \right)
:=
\prod_{k\in\underline{\kappa}}\ev_k^*(L_k)\cdot \prod_{i\in \underline{n}} \ev_i^*\left( p_i\right)\cdot\prod_{j'=1}^{l'} \ft_{\mu_{j'}}^*\left( |\mu'_{j'}|\right)\cdot\prod_{\tilde{j}=1}^{\tilde{l}} \ft_{\lambda_{\tilde{j}}}^*\left( 0\right) \cdot \mathcal{M}_{0,m}\left(\mathbb{R}^2,d\right).
\end{align*}
Each point $p_{i}\in p_{\underline{n}}$ is a $2$-dimensional condition. Each multi line $L_k\in L_{\underline{\kappa}}$ and each cross-ratio $\mu'_{j'}\in \mu'_{[l']}$, $\lambda_{\tilde{j}}\in \lambda_{[\tilde{l}]}$ is a $1$-dimensional condition. Hence the dimension of $Z_{d}\left(p_{\underline{n}},L_{\underline{\kappa}}, \lambda_{[\tilde{l}]},\mu'_{[l']} \right)$ is
$(3d-1+m)-(2\cdot \#\underline{n}+\tilde{l}+l'+\#\underline{\kappa})$, where $3d-1+m$ is the dimension of $\mathcal{M}_{0,m}\left(\mathbb{R}^2,d \right)$.
\end{definition}

Notice that each tropical stable map in $Z_{d}\left(p_{\underline{n}},L_{\underline{\kappa}}, \lambda_{[\tilde{l}]},\mu'_{[l']} \right)$ has $3$ different kinds of contracted ends, namely contracted ends with labels in $\underline{n}$ that satisfy point conditions, contracted ends with labels in $\underline{\kappa}$ that satisfy line conditions and contracted ends with labels in $\underline{f}$ that satisfy no point or line conditions. Given $\underline{n}$ and $\underline{\kappa}$, we can calculate $\#\underline{f}$ using
\begin{align*}
m=\#\underline{n}+\#\underline{\kappa}+\#\underline{f}.
\end{align*}

\begin{definition}[General position]\label{def:general_pos}
Let $p_{\underline{n}},L_{\underline{\kappa}}, \lambda_{[\tilde{l}]},\mu'_{[l']}$ be conditions as in Definition \ref{def:cycle_Z_d} such that
\begin{align}\label{eq:gen_pos_condition_count}
3d-1=\#\underline{n}+\tilde{l}+l'-\#\underline{f}
\end{align}
holds.
These conditions are in \textit{general position} if $Z_{d}\left(p_{\underline{n}},L_{\underline{\kappa}}, \lambda_{[\tilde{l}]},\mu'_{[l']} \right)$ is a zero-dimensional nonzero cycle that lies inside top-dimensional cells of $\prod_{\tilde{j}=1}^{\tilde{l}} \ft_{\lambda_{\tilde{j}}}^*\left( 0\right)\cdot\mathcal{M}_{0,m}\left(\mathbb{R}^2,d\right)$.
\end{definition}

\begin{definition}\label{def:numbers_of_interest}
For general positioned condition as in Definition \ref{def:general_pos}, where we additionally require from the cross-ratios that no label of a non-contracted end appears in any of the cross-ratios $\lambda_{[\tilde{l}]},\mu'_{[l']}$, we define
\begin{align*}
N_{d}\left(p_{\underline{n}},L_{\underline{\kappa}}, \lambda_{[\tilde{l}]},\mu'_{[l']} \right)
:=
\deg\left( Z_{d}\left(p_{\underline{n}},L_{\underline{\kappa}}, \lambda_{[\tilde{l}]},\mu'_{[l']} \right) \right),
\end{align*}
where $\deg$ is the degree function that sums up all multiplicites of the points in the intersection product. In other words, $N_{d}\left(p_{\underline{n}},L_{\underline{\kappa}}, \lambda_{[\tilde{l}]},\mu'_{[l']} \right)$ is the number of rational tropical stable maps to $\mathbb{R}^2$ (counted with multiplicity) of degree $d$ satisfying the cross-ratios $\lambda_{[\tilde{l}]},\mu'_{[l']}$, the multi line conditions $L_{\underline{\kappa}}$ and point conditions $p_{\underline{n}}$.
\end{definition}

\begin{remark}
Allowing only tropical multi lines in Definition \ref{def:numbers_of_interest} instead of arbitrary rational tropical curves is not a restriction, since we can always pass to the recession fan of a tropical curve without effecting the count, see Remark \ref{remark:recession_fan} and \cite{Allermann}.
\end{remark}

\begin{remark}\label{remark:numbers_independent_of_positions_and_lengths}
The numbers $N_{d}\left(p_{\underline{n}},L_{\underline{\kappa}}, \lambda_{[\tilde{l}]},\mu'_{[l']} \right)$ are independent of the exact position of points $p_{\underline{n}}$ and multi lines $L_{\underline{\kappa}}$ as long as the set of all conditions is in general position. Moreover, the numbers are also independent of the exact lengths $|\mu'_{1}|,\dots,|\mu'_{l'}|$ of the non-degenerated cross-ratios. In particular, 
\begin{align*}
N_{d}\left(p_{\underline{n}},L_{\underline{\kappa}}, \lambda_{[\tilde{l}]},\mu'_{[l']} \right)
=
N_{d}\left(p_{\underline{n}},L_{\underline{\kappa}}, \lambda_{[\tilde{l}]},\mu_{[l']} \right),
\end{align*}
where $\mu_{j'}$ is the degeneration of $\mu'_{j'}$.
\end{remark}

Given a tropical stable map $C$ that satisfies a cross-ratio condition $\lambda'$, we can think of this condition as a path of fixed length $|\lambda'|$ inside $C$. Thus a degenerated cross-ratio condition $\lambda$ can be thought of as a path of length zero inside a tropical stable map, i.e. there is a vertex of valence $>3$ in $C$ satisfying a degenerated cross-ratio. Or in other words, there is a vertex $v\in C$ such that the image of $v$ under $\ft_{\lambda}$ is $4$-valent. We say that $\lambda$ is \textit{satisfied at $v$}. Obviously, \textit{a tropical stable map $C$ satisfies a degenerated cross-ratio condition if and only if there is a vertex of $C$ that satisfies the degenerated cross-ratio}. We define the set $\lambda_v$ of cross-ratios associated to a vertex $v$ that consists of all given cross-ratios whose images of $v$ using the forgetful map are $4$-valent.

\begin{remark}\label{remark:path_criterion}
An equivalent and more descriptive way of saying that a cross-ratio is satisfied at a vertex is the \textit{path criterion}:
Let $C$ be a tropical stable map and let $\lambda=\lbrace \beta_1,\dots,\beta_4\rbrace$ be a cross-ratio, then a pair $\left( \beta_i,\beta_j\right)$ induces a unique path in $C$. If the paths associated to $\left( \beta_{i_1},\beta_{i_2} \right)$ and $\left( \beta_{i_3},\beta_{i_4} \right)$ intersect in exactly one vertex $v$ of $C$ for all pairwise different choices of $i_1,\dots,i_4$ such that $\lbrace i_1,\dots,i_4 \rbrace =\lbrace 1,\dots,4 \rbrace$, then and only then the cross-ratio $\lambda$ is satisfied at $v$. Note that ``for all choices" above is equivalent to ``for one choice".
\end{remark}

\begin{construction}\label{constr:resolution_of_vertex}
Let $v$ be a vertex of an abstract tropical curve and $\lambda_j\in\lambda_v$. 
We say that $v$ is \textit{resolved according to $\lambda'_j$} (where $\lambda'_j$ is a cross-ratio that degenerates to $\lambda_j$) if
the equality
\begin{align*}
\val(v)=3+\#\lambda_v 
\end{align*}
holds, $v$ is replaced by two vertices $v_1,v_2$ that are connected by a new edge such that $\lambda'_j$ is satsfied,
\begin{align*}
\lambda_v=\lbrace\lambda_j\rbrace\cupdot\lambda_{v_1}\cupdot\lambda_{v_2}
\end{align*}
is a union of pairwise disjoint sets and
\begin{align*}
\val(v_k)=3+\#\lambda_{v_k}
\end{align*}
holds for $k=1,2$.
\end{construction}

\begin{figure}[H]
\centering
\def\svgwidth{280pt}
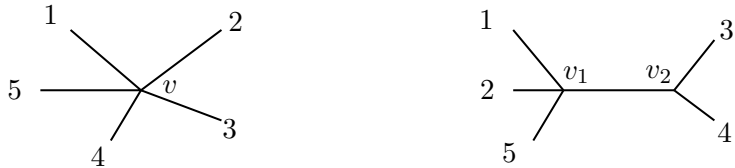
\caption{Let $\lambda_{1}:=\lbrace 1,2,3,4\rbrace$ and $\lambda_{2}:=\lbrace 1,2,3,5\rbrace$ be two degenerated cross-ratios. On the right there is a $5$-valent vertex $v$ with $\lambda_{v}=\lbrace \lambda_1,\lambda_2\rbrace$. On the left $v$ is resolved according to $\lambda'_1:=(12|34)$. Notice that the resolution is unique in this case.}
\label{Figure26}
\end{figure}

\begin{definition}[Cross-ratio multiplicity]\label{def:CR_mult}
Let $v$ be a $(3+\#\lambda_v)$-valent vertex of an abstract tropical curve with $\lambda_v=\lbrace \lambda_{j_1},\dots,\lambda_{j_r}\rbrace$ and let $\lambda'_{j_t}$ be cross-ratios that degenerate to $\lambda_{j_t}$ for $t=1,\dots,r$ such that $|\lambda'_{j_1}|>\dots >|\lambda'_{j_r}|$.
A \textit{total resolution} of $v$ is a $3$-valent labeled abstract tropical curve on $r$ vertices that arises from $v$ by resolving $v$ according to the following recursion. First, resolve $v$ according to $\lambda'_{j_1}$. The two new vertices are denoted by $v_1,v_2$. Choose $v_k$ with $\lambda_{j_2}\in\lambda_{v_k}$ and resolve it according to $\lambda_{j_2}'$ (this may not be unique, pick one resolution). Now we have $3$ vertices $v_1,v_2,v_3$ from which we pick the one with $\lambda_{j_3}\in\lambda_{v_k}$, resolve it and so on. We define the \textit{cross-ratio multiplicity} $\mult_{\CR}(v)$ of $v$ to be the number of total resolution of $v$. This number does not depend on the choice of non-degenerated cross-ratios $\lambda'_{j_1},\dots,\lambda'_{j_r}$, in particular, it does not depend on the order $|\lambda'_{j_1}|>\dots >|\lambda'_{j_r}|$, see \cite{CR1}. In the special case of $\#\lambda_v=0$, we set $\mult_{\CR}(v)=1$.
\end{definition}

\begin{example}
Let $v$ be a $6$-valent vertex such that $\lambda_v=\lbrace \lambda_1,\lambda_2,\lambda_3\rbrace$ and the degenerated cross-ratios are given by $\lambda'_1:=(12|56),\lambda'_2:=(34|56),\lambda'_3=(12|34)$. The following two $3$-valent trees schematically show all total resolutions of $v$ with respect to $|\lambda'_1|>|\lambda'_2|>|\lambda'_3|$.
\begin{figure}[H]
\centering
\def\svgwidth{300pt}
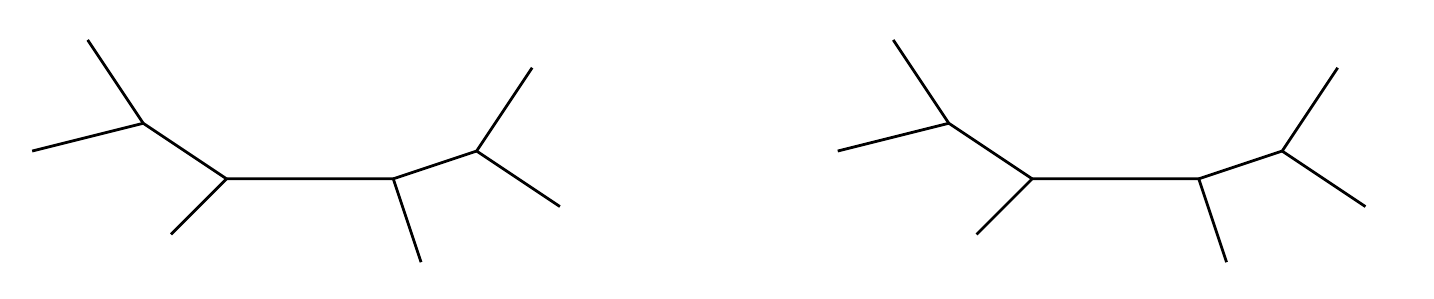
\label{Figure10}
\end{figure}
\end{example}

\begin{open problem}
The numbers $\mult_{\CR}(v)$ are not well understood. Of course, one can calculate them by considering all trees  with an appropriate number of labeled ends and pick the ones that are total resolutions of $v$ with respect to the given cross-ratios. This approach is neither fast nor pleasing. So a question naturally comes up: is there another, more efficient way to calculate the cross-ratio multiplicity $\mult_{\CR}(v)$ of a vertex $v$ satisfying degenerated cross-ratios?
\end{open problem}

\begin{definition}[Evaluation multiplicity]\label{def:ev_mult}
Let $C$ be a tropical stable map that contributes to $N_{d}\left(p_{\underline{n}},L_{\underline{\kappa}}, \lambda_{[l]} \right)$. Consider the \textit{ev-matrix} $M(C)$ of $C$, which is given by the locally (around $C$) linear map $\bigtimes_{t\in \underline{n}\cup \underline{\kappa}} \ev_t:\mathcal{M}_{0,m}\left(\mathbb{R}^2,d \right)\to\mathbb{R}^{2\cdot \#\underline{n}+\#\underline{\kappa}}$, where the coordinates on $\mathcal{M}_{0,m}\left(\mathbb{R}^2,d \right)$ are the bounded edges' lengths.
The \textit{evaluation multiplicity} $\mult_{\ev}(C)$ of $C$ is defined by
\begin{align*}
\mult_{\ev}(C):= |\det(M(C))|.
\end{align*}
The matrix in Example \ref{ex:1/1_split} provides an example of an $\ev$-matrix
\end{definition}

\begin{proposition}[\cite{CR1}]\label{prop:cr_mult}
If $C$ is a tropical stable map that contributes to $N_{d}\left(p_{\underline{n}},L_{\underline{\kappa}}, \lambda_{[l]} \right)$, then the multiplicity $\mult(C)$ with which $C$ contributes to this intersection product is given by
\begin{align*}
\mult(C)=\mult_{\ev}(C)\prod_{v\mid v \textrm{ vertex of }C}\mult_{\CR}(v),
\end{align*}
where $\mult_{\ev}(C)$ is the absolute value of the determinant of the $\ev$-matrix associated to $C$, see \cite{Rau, CR1}.
\end{proposition}

\begin{corollary}[\cite{CR1}]\label{cor:CR_pfade_ueber_alle_edges_an_vertex}
Let $C$ be a tropical stable map that contributes to $N_{d}\left(p_{\underline{n}},L_{\underline{\kappa}}, \lambda_{[l]} \right)$. Let $v\in C$ be a vertex of $C$ such that $\val(v)>3$. Then for every edge $e$ adjacent to $v$ in $C$ there is an entry $\beta_i$ in some $\lambda_j\in\lambda_v$ such that $e$ is in the shortest path from $v$ to the end labeled with $\beta_i$. 
\end{corollary}

The following correspondence theorem allows us to obtain non-tropical results from our tropical ones in case of no multi line conditions.

\begin{theorem}[Correspondence Theorem 5.1 of \cite{IlyaCRC}]\label{thm:correspondence_thm_CRC}
Let $N^{\operatorname{class}}_{d}\left(p_{\underline{n}}, \mu_{[l]} \right)$ denote the number of plane rational degree $d$ curves that satisfy point conditions and non-tropical cross-ratios $\mu_1,\dots,\mu_l$ such that all conditions are in general position. Then
\begin{align*}
N^{\operatorname{class}}_{d}\left(p_{\underline{n}}, \mu_{[l]} \right)
=
N_{d}\left(p_{\underline{n}}, \lambda'_{[l]} \right)
\end{align*}
holds, where $\lambda'_j$ is the tropical cross-ratio associated to $\mu_j$ for $j\in [l]$ in the sense of \cite{IlyaCRC}.
\end{theorem}

\section{Splitting curves with cross-ratios}

\subsection*{Existence of contracted bounded edges}
The aim of this subsection is to prove Propositions \ref{prop:contracted_bounded_edge}, \ref{prop:contracted_bounded_edge_no_point_conditions}, which are crucial for the recursion we aim for. They guarantee that the tropical stable maps we are dealing with have a contracted bounded edge at which we can split them. Proposition \ref{prop:contracted_bounded_edge} covers the case where we have at least one point condition. Proposition \ref{prop:contracted_bounded_edge_no_point_conditions} covers the case of no point conditions.

\begin{proposition}\label{prop:contracted_bounded_edge}
Let $n\geq 1$ and let $C$ be a tropical stable map contributing to $N_{d}\left(p_{\underline{n}},L_{\underline{\kappa}}, \lambda_{[l-1]},\lambda'_l \right)$, where $\lambda'_l$ is a non-degenerated tropical cross-ratio. If $|\lambda'_l|$ is large, then there is exactly one contracted bounded edge  in $C$.
\end{proposition}

To keep track of the overall structure of the proof of Proposition \ref{prop:contracted_bounded_edge}, we briefly outline important steps:
\begin{itemize}
\item
Definition \ref{def:movable_component}: Forget $\lambda'_l$, to obtain a $1$-dimensional cycle $Y$.
\item
Definition \ref{def:movable_component}, Remark \ref{remark:change_of_movement_direction}, Example \ref{ex:direction_of_movement_changes}: Consider the $1$-dimensional rays of $Y$. They correspond to tropical curves $\Gamma$ that satisfy $p_{\underline{n}},L_{\underline{\kappa}}, \lambda_{[l-1]}$ such that $\Gamma$ admits a movement which gives rise to an unbounded $1$-dimensional family of curves of the same combinatorial type as $\Gamma$. Hence we should study tropical curves $\Gamma$ that have a \textit{movable component} (i.e. a subgraph) $B$ which can be moved unboundedly without changing the combinatorial type of $\Gamma$.
\item
Definition \ref{def:partial_order}, Corollary \ref{corollary:no_chains_strong}: Show that $B$ contains a single vertex. For this, we define \textit{chains} of vertices in $B$ and show that no chain has more than one element.
\item
Proof of Proposition \ref{prop:contracted_bounded_edge}: Conclude that there must be a contracted bounded edge.
\end{itemize}

\begin{definition}[Movable component]\label{def:movable_component}
Let $\Gamma$ be a tropical curve with no contracted bounded edge coming from a stable map in the $1$-dimensional cycle (for notation, see Definition \ref{def:general_pos})
\begin{align*}
Y:=\prod_{k\in\underline{\kappa}}\ev_k^*(L_k)\cdot \prod_{i\in\underline{n}} \ev_i^*\left( p_i\right)\cdot \prod_{j=1}^{l-1} \ft_{\lambda_j}^*\left( 0\right) \cdot \mathcal{M}_{0,m}\left(\mathbb{R}^2,d\right)
\end{align*}
such that $\Gamma$ gives rise to a $1$-dimensional family of curves by moving some of its vertices. Since the family obtained by moving vertices of $\Gamma$ is $1$-dimensional, no vertex can be moved freely, i.e. in each possible direction. Hence each vertex of $\Gamma$ is either \textit{fixed}, i.e. it can not be moved at all, or \textit{movable} in a direction given by a vector in $\mathbb{R}^2$ which we call \textit{direction of movement} of $v$. Directions of movements of vertices are indicated in Figure \ref{Figure_23} of Example \ref{ex:direction_of_movement_changes}.
Since each movable vertex $v$ cannot move freely, its movement is restricted by a condition imposed to it via an edge adjacent to $v$.
More precisely, $v$ either needs to be adjacent to a fixed vertex or to a contracted end which satisfies a multi line condition.
The connected component of $\Gamma$ which consists of all movable vertices of $\Gamma$ (and edges connecting movable vertices) is called the \textit{movable component} of $\Gamma$. Notice that there is exactly one movable component since $\Gamma$ gives rise to a $1$-dimensional family only.
A connected component of $\Gamma$ that is obtained from $\Gamma$ by removing the movable component is called \textit{fixed component}. We say that a movable component allows an \textit{unbounded movement}, if the movement of the movable component gives rise to a family of curves of the same combinatorial type as $\Gamma$ that is unbounded.
\end{definition}

\begin{remark}\label{remark:change_of_movement_direction}
Consider a $1$-dimensional family of curves of the same combinatorial type that is unbounded and a movable component within some curve of this family that allows an unbounded movement. Notice that the direction of movement $b$ of a vertex $v$ in this movable component might change as moving the component generates the family. Since $v$ is either adjacent to a fixed vertex or adjacent to an end satisfying a multi line condition, $b$ can only change, when $v$ is adjacent to an end that satisfies a multi line condition $L$. Thus $b$ can only change if $v$ passes over the vertex of $L$, see Example \ref{ex:direction_of_movement_changes}. Hence the direction of movement of a vertex in the movable component cannot change if we already moved the movable component enough. In the following we focus on movable components that allow an unbounded movement and that already have been moved sufficiently such that we can assume that the direction of movement of each vertex therein does not change anymore when moving. In particular, we may assume that the direction of movement of a vertex satisfying a multi line condition is parallel to $(-1,0),(0,-1)$ or $(1,1)$.
\end{remark}

\begin{example}\label{ex:direction_of_movement_changes}
Figure \ref{Figure_23} provides an example of a curve $C$ in $\mathbb{R}^2$ whose contracted ends labeled with $1,2,4,5$ satisfy point conditions and the contracted end labeled with $3$ satisfies a multi line condition (the dashed line). The vertex $v$ adjacent to the end labeled with $3$ is in the movable component of $C$ and the direction of movement $b$ (indicated by an arrow) of $v$ might changes as $v$ is moved. The movement shown in Figure \ref{Figure_23} is bounded.

\begin{figure}[H]
\centering
\def\svgwidth{450pt}
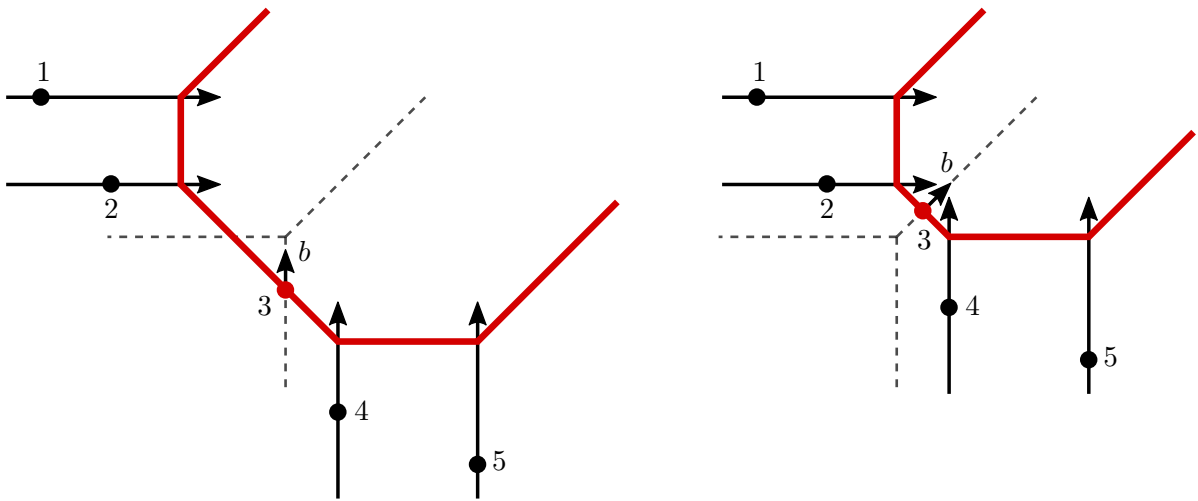
\caption{A curve satisfying point conditions and one multi line condition. The movable component is drawn in bold red. The arrows indicate the directions of movement. The movement shown is bounded.}
\label{Figure_23}
\end{figure}

\end{example}

\begin{remark}
Showing that $B$ contains a single vertex is non-trivial. However, the difficulties arise primarily due to the cross-ratios. If we have no cross-ratios and thus every vertex in our tropical curves is $3$-valent, then the movable component boils down to a string as introduced in \cite{KontsevichPaper}, which can be thought of as a single chain.
\end{remark}

\begin{classification}[Types of movable vertices]\label{classification:types_of_vertices}
Let $\Gamma$ be a tropical curve as in Definition \ref{def:movable_component}. If there is a vertex $v$ in the movable component of $\Gamma$ that is adjacent to a fixed component and all of its adjacent edges and ends which are non-contracted are parallel, then the movable component of $\Gamma$ has exactly one vertex, namely $v$. Otherwise $\Gamma$ would not give rise to a $1$-dimensional family only.

Hence the following classification is complete if we assume that the movable component of $\Gamma$ has more than $1$ vertex (if it has exactly $1$ vertex, then we can directly jump to the proof of Proposition \ref{prop:contracted_bounded_edge}): We distinguish $4$ types of vertices in the movable component.

\begin{tabular}{lp{13cm}}
\textbf{Type (I)} & vertices are adjacent to a fixed component and not all adjacent edges and non-contracted ends are parallel.\\
\textbf{Type (II)} & vertices are not $3$-valent and adjacent to a contracted end which satisfies a multi line condition.\\
\textbf{Type (IIIa)} & vertices are $3$-valent, adjacent to two bounded edges and adjacent to a contracted end which satisfies a multi line condition\\
\textbf{Type (IIIb)} & vertices are $3$-valent, adjacent to one bounded edge, a contracted end which satisfies a multi line condition and an end in standard direction.\\
\end{tabular}
\,
Throughout this section we use the assumption that the movable component of $\Gamma$ has more than $1$ vertex whenever we refer to this classification of vertices.
\end{classification}

\begin{construction}\label{constr:tilde_Gamma}
In the following we often forget the vertices of type (IIIa) and type (IIIb) in $\Gamma$ by gluing the non-contracted edges adjacent to a vertex of type (IIIa) (resp. type (IIIb)) together and obtain a tropical curve denoted by $\tilde{\Gamma}$. \textit{We fix this notation of $\tilde{\Gamma}$ throughout this section}.

If $\tilde{\Gamma}$ allows no $1$-dimensional movement, then the only vertices in the movable component of $\Gamma$ are of type (IIIa) or (IIIb). Hence there is no type (I) vertex in the movable component of $\Gamma$. Thus $\Gamma$ has no fixed component. In particular $p_{[n]}=\emptyset$, but this case is treated separately in Lemma \ref{lemma:initial_values_1}, Lemma \ref{lemma:initial_values_2} and Proposition \ref{prop:contracted_bounded_edge_no_point_conditions}. Therefore we can assume that $\tilde{\Gamma}$ allows an unbounded $1$-dimensional movement.
\end{construction}

\begin{figure}[H]
\centering
\def\svgwidth{220pt}
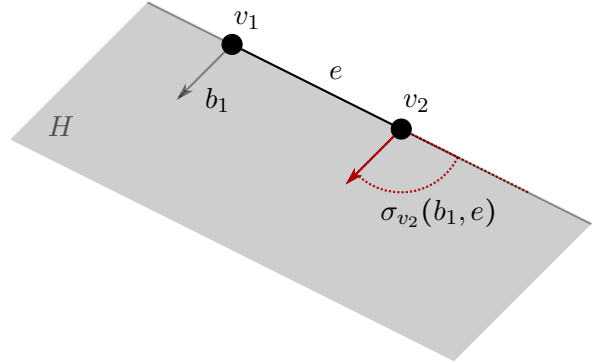
\caption{The cone $\sigma_{v_2}(b_1,e)$ in which the direction of movement of $v_2$ lies. The slope of the edge connecting $v_1,v_2$ is fixed during the movement. Hence the translation $b_2+v_2$ of the direction of movement $b_2$ of $v_2$ is contained in the open half-plane $H$ whose boundary is $\langle e \rangle + v_2$ and whose interior contains $b_1+v_1$.}
\label{Figure_2}
\end{figure}

\begin{lemma}[Angle Lemma]\label{lemma:angle_lemma}
Let $\tilde{\Gamma}$ be a tropical curve in $\mathbb{R}^2$ as in Construction \ref{constr:tilde_Gamma} that allows an unbounded $1$-dimensional movement.
Let $v_1,v_2$ be adjacent vertices in the movable component of $\tilde{\Gamma}$, let $b_1\neq 0$ be the direction of movement of $v_1$ and let $v(e,v_1)\neq b_1$ be the direction vector at $v_1$ of the edge $e$ that connects $v_1$ and $v_2$. Then the direction of movement $b_2$ of $v_2$ lies in the half-open cone
\begin{align*}
 \sigma_{v_2}(b_1,e):=\lbrace x\in\mathbb{R}^2 \mid x=v_2+\lambda_{1} v(e,v_1)+\lambda_{2} b_{1}, \; \lambda_{1} \in\mathbb{R}_{\geq 0},\; \lambda_{2}\in\mathbb{R}_{>0} \rbrace
\end{align*}
centered at $v_2$ that is spanned by $b_1$ and $v(e,v_1)$, where half-open means that the boundary of $\sigma_{v_2}(b_1,e)$ that is generated by $b_1$ is part of the cone and the boundary that is generated by $v(e,v_1)$ is not part of the cone, while $v_2$ itself is also not part of the cone.
\end{lemma}

\begin{proof}
This is true since the length of the edge $e'$ that connects $v_1$ and $v_2$ cannot shrink when moving $v_1$ and $v_2$, otherwise the movement would be bounded. Therefore the (affine) lines $\langle b_1 \rangle+v_1$ and $\langle b_2\rangle+v_2$ must either be parallel or their point of intersection does not lie in $H$.
\end{proof}

\begin{definition}[Partial order]\label{def:partial_order}
We use the notation from Construction \ref{constr:tilde_Gamma}. Let $\tilde{\Gamma}$ be a tropical curve in $\mathbb{R}^2$ that allows an unbounded $1$-dimensional movement and let $H$ be an open half-plane. If we translate $H$ to a vertex $v\in\tilde{\Gamma}$, i.e. $v$ is contained in the boundary of $H$, then we denote the translated half-plane by $H_v$. Let $M$ be the set of all vertices of the movable component of $\tilde{\Gamma}$, i.e. $M$ consists of all type (I) and type (II) vertices of the movable component of $\Gamma$. The half-plane $H$ induces a partial order $\Omega(H)$ on $M$ as follows: For $v_1,v_2\in M$ define
\begin{align*}
v_1\geq v_2 :\iff 
\begin{cases}
v_1=v_2, \textrm{ or} \\
v_2 \textrm{ is adjacent to } v_1\textrm{ and } v_2\in H_{v_{1}}.
\end{cases}
\end{align*}
Here, we only use open half-planes $H$ such that $b_1+v_1\in H_{v_1}$. Therefore if $v_1\geq \dots \geq v_n$ is a maximal chain and $b_i$ is the direction of movement of $v_i$ for $i=1,\dots,n$, then $b_i+v_i\in H_{v_i}$ for $i=1,\dots,n$ by inductively applying Lemma \ref{lemma:angle_lemma}.

\end{definition}

\begin{notation}\label{notation:direction_of_movement}
Given a chain $v_1\geq \dots \geq v_n$ in the movable component of $\tilde{\Gamma}$, we denote the direction of movement of $v_i$ by $b_i$ for $i=1,\dots, n$ throughout this section. If such a chain is maximal, then an edge connecting $v_i$ and $v_{i+1}$ is usually denoted by $e_i$ for $i=1,\dots, n-1$. by abuse of notation, we often write $e_i$ instead of the direction vector $v(e_i,v_i)$ at $v_i$ from Definition \ref{definition:moduli_stable_maps}.
\end{notation}

\begin{figure}[H]
\centering
\def\svgwidth{400pt}
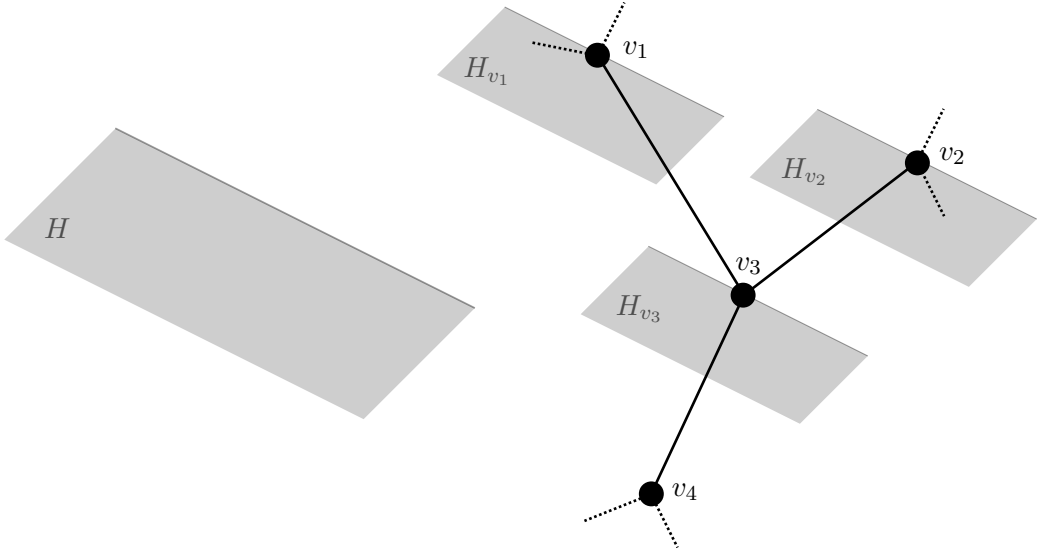
\caption{This is an example of the partial order $\Omega(H)$ for $H\subset\mathbb{R}^2$ which is an open half-plane as shown on the left (the boundary of the half-plane is darkened). On the right there is a sketch of a tropical curve in $\mathbb{R}^2$ such that $v_1\geq v_3 \geq v_4$ and $v_2\geq v_3 \geq v_4$ with respect to the order $\Omega(H)$.}
\label{Figure_6}
\end{figure}

\begin{lemma}[Maximal chains]\label{lemma:maximal_chains}
We use Notation \ref{notation:direction_of_movement}. Let $\tilde{\Gamma}$ be a tropical curve in $\mathbb{R}^2$ as in Construction \ref{constr:tilde_Gamma}, that allows an unbounded $1$-dimensional movement.
Let $v_1\geq\dots\geq v_n$ be a maximal chain with $n>1$ and $b_1+v_1\in H_{v_1}$ in $\tilde{\Gamma}$ with respect to $\Omega(H)$ as in Definition \ref{def:partial_order}. Then there is no vertex $v_{n+1}\in \tilde{\Gamma}$ adjacent to $v_n$ such that $v_{n+1}\in H_{v_n}$.
\end{lemma}

\begin{proof}
We use Notation \ref{notation:direction_of_movement}.
By definition, $v_n,b_{n-1}+v_{n-1}\in H_{v_{n-1}}$ and there is an edge $e_{n-1}$ connecting $v_{n-1}$ to $v_n$. If $\langle b_{n-1} \rangle = \langle e_{n-1} \rangle$, then $b_{n-1}$ and $b_n$ are parallel. Thus we have a $2$-dimensional movement which yields a contradiction since we just allow a $1$-dimensional movement. In total, the requirements of Lemma \ref{lemma:angle_lemma} are fulfilled such that $b_n+v_n\in H_{v_n}$ follows. Since there is an edge $e_n$ that connects $v_n$ to $v_{n+1}$ and $v_{n+1}\in H_{v_n}$, Definition \ref{def:partial_order} yields $v_{n}\geq v_{n+1}$ with respect to $\Omega(H)$. This contradicts our maximality assumption.
\end{proof}

%

\begin{definition}[Special half-planes]\label{definition:special_half-planes}
Let $e\in\mathbb{R}^2$ be a vector of standard direction, i.e. $e$ is in $\lbrace (-1,0),(0,-1),(1,1)\rbrace$. An open half-plane is called \textit{special half-plane} if the affine subspace $\langle e \rangle + v\subset \mathbb{R}^2$ for some $v\in\mathbb{R}^2$ that is generated by $e$ is the boundary of $H$. There are six special half-planes up to translation, see Figure \ref{Figure_27}.
\end{definition}

\begin{figure}[H]
\centering
\def\svgwidth{330pt}
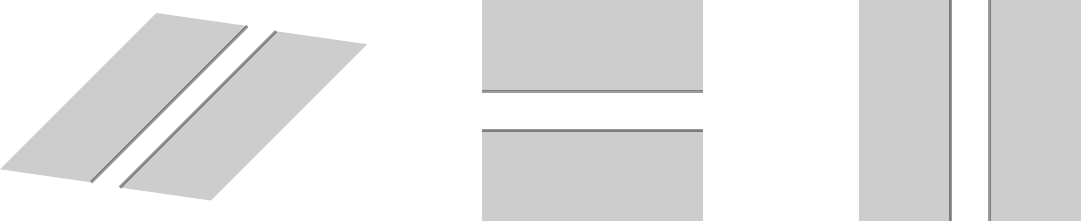
\caption{All six special half-planes up to translation. The boundary of each is darkened.}
\label{Figure_27}
\end{figure}

%

\begin{definition}
An open half-plane $H$ is called $1$-ray (resp. $2$-ray) half-plane if it contains exactly one (resp. two) rays of standard direction. Notice that special half-planes are $1$-ray half-planes.
\end{definition}

\begin{lemma}\label{lemma:about_chains_weak_vorarbeit}
Let $\tilde{\Gamma}$ be a tropical curve in $\mathbb{R}^2$ as in Construction \ref{constr:tilde_Gamma} that allows an unbounded $1$-dimensional movement. Let $v_1$ be a vertex of the movable component of $\tilde{\Gamma}$. Let $H$ be a $1$-ray half-plane that contains a ray of standard direction $D$. If $v_1\geq\dots\geq v_n$ is a maximal chain starting at $v_1$ with respect to $\Omega(H)$ such that $n>1$ and $b_1+v_1\in H_{v_1}$, then there is an end $e$ of $\tilde{\Gamma}$ adjacent to $v_n$ which is parallel to $D$.
\end{lemma}

\begin{proof}
We use Notation \ref{notation:direction_of_movement}.
Notice that $v_{n-1}\geq v_n$. Hence $e_{n-1}+v_n \nin \overline{H}_{v_n}$, where $\overline{H}_{v_n}$ denotes the closure of $H_{v_n}$. Thus by balancing, there is an edge $e\in\tilde{\Gamma}$ adjacent to $v_n$ such that $e\in H_{v_n}$.
If $e$ connects $v_n$ to a fixed component, then $b_n+v_n\nin \overline{H}_{v_n}$ because the movement of $v_n$ should be unbounded, i.e. $b_n$ moves $v_n$ away from that fixed component while $\langle e \rangle +v_n = \langle b_n \rangle+v_n$, which contradicts that $b_n+v_n\in H_{v_n}$ by Lemma \ref{lemma:angle_lemma}.
Hence $e$ is an end of $\tilde{\Gamma}$ by Lemma \ref{lemma:maximal_chains}. Since $H_{v_n}$ is a $1$-ray half-plane containing exactly $1$ ray of standard direction $D$, the direction of $e$ is $D$.
\end{proof}

\begin{lemma}[About maximal chains, weak version]\label{lemma:about_chains_weak}
Let $\tilde{\Gamma}$ be a tropical curve in $\mathbb{R}^2$ as in Construction \ref{constr:tilde_Gamma} that allows an unbounded $1$-dimensional movement. Let $v_1$ be a vertex of the movable component of $\tilde{\Gamma}$. If there is a $1$-ray half-plane $H$ and $v_1\geq\dots\geq v_n$ is a maximal chain starting at $v_1$ with respect to $\Omega(H)$ such that $n>1$ and $b_1+v_1\in H_{v_1}$, then $v_n$ is a $3$-valent type (I) vertex.
\end{lemma}

\begin{proof}
We use Notation \ref{notation:direction_of_movement}.
By Lemma \ref{lemma:about_chains_weak_vorarbeit} there is an end $e$ of $\tilde{\Gamma}$ adjacent to $v_n$. Moreover, since $H$ is a $1$-ray half-plane containing exactly $1$ ray of standard direction $D$, the direction of $e$ is $D$.
Assume that the valency of $v_n$ is greater than $3$, i.e. there is a cross-ratio in $\lambda_{v_n}$. Since all cross-ratios have only labels of contracted ends as entries (see Definition \ref{def:numbers_of_interest}), we can apply Corollary \ref{cor:CR_pfade_ueber_alle_edges_an_vertex}. Therefore there is a vertex $v\in\Gamma$ connected to $v_n$ via $e$ such that $v$ is of type (IIIa) or type (IIIb) such that $v$ satisfies a multi line condition. Since the movement of $v$ is unbounded, its direction of movement, denoted by $b$, is parallel to $e$ (cf. Remark \ref{remark:change_of_movement_direction}). Therefore the movable component of $\Gamma$ allows a $2$-dimensional movement, which is a contradiction.

In total, $v_n$ can only be a $3$-valent type (I) vertex since we ruled out the other cases.
\end{proof}

\begin{corollary}\label{corollary:about_chains_weak}
If we make the same assumptions as in Lemma \ref{lemma:about_chains_weak} and additionally require that $H$ is a special half-plane (see Definition \ref{definition:special_half-planes}), then there exists no chain $v_1\geq \dots \geq v_n$ with respect to $\Omega(H)$ such that $n>1$ and $b_1+v_1\in H_{v_1}$.
\end{corollary}

\begin{proof}
We use Notation \ref{notation:direction_of_movement}.
It is sufficient to show the statement for maximal chains $v_1\geq \dots \geq v_n$ starting at $v_1$. So we assume that our chain is maximal. The vertex $v_n$ is $3$-valent of type (I) by Lemma \ref{lemma:about_chains_weak}. Let $D$ denote the ray of standard direction that is contained in $H$. 
By Lemma \ref{lemma:about_chains_weak_vorarbeit}, there is an end $e$ adjacent to $v_n$ of standard direction $D$. Denote the edge that connects $v_n$ to a fixed component by $f$, and because $\langle f \rangle +v_n = \langle b_n \rangle +v_n$, we know that $f+v_n\nin\overline{H}_{v_n}$. Since all ends are of weight $1$, the end $e$ is also of weight $1$. Using balancing and the definition of special half-planes, we conclude that the edge $e_{n-1}$ that connects $v_{n-1}$ to $v_n$ lies in the boundary of $H_{v_n}$, which contradicts $v_{n-1}\geq v_n$.
\end{proof}

\begin{observation}\label{observation:half-planes}
Let $v_1,v_2$ be two vertices of the movable component of $\tilde{\Gamma}$. Let $e$ be an edge that connects $v_1$ and $v_2$ and let $b_1$ be the direction of movement of $v_1$. Corollary \ref{corollary:about_chains_weak} shows that there cannot be an open half-plane $H$ such that $b_1+v_1,e+v_1\in H_{v_1}$, and such that $H_{v_1}$ is a special half-plane. Note that $\langle b_1 \rangle \neq \langle e \rangle$, otherwise our movable component would move in a $2$-dimensional way. Therefore, for each pair of directions of $b_1$ and $e$, there are open half-planes that contain $b_1$ and $e$. But each of these open half-planes is not a special half-plane. This observation gives rise to the following classification.
\end{observation}

\begin{classification}[Dependence of $b_1$ and $e$]\label{classification:types_of_vertices_2}
Let $\tilde{\Gamma}$ be as in Construction \ref{constr:tilde_Gamma}. In particular, we assume that $\tilde{\Gamma}$ has more than one vertex. Use the notation of Observation \ref{observation:half-planes}, i.e. let $v_1\in\tilde{\Gamma}$ be a vertex with direction of movement $b_1$. If $b_1+v_1$ is in one of the dashed red cones in Figure \ref{Figure_7}, then $e+v_1$ has to lie in the opposite cone. Otherwise there would be a special half-plane $H$ such that $b_1+v_1,e+v_1\in H_{v_1}$, which contradicts Observation \ref{observation:half-planes}. We distinguish the $3$ cases depicted in Figure \ref{Figure_7}: If $b_1+v_1$ and $e+v_1$ lie in the red cones depicted on the left, then $v_1$ is said to be of type $F_1$. The other two cases can be seen in Figure \ref{Figure_7}.

\begin{figure}[H]
\centering
\def\svgwidth{400pt}
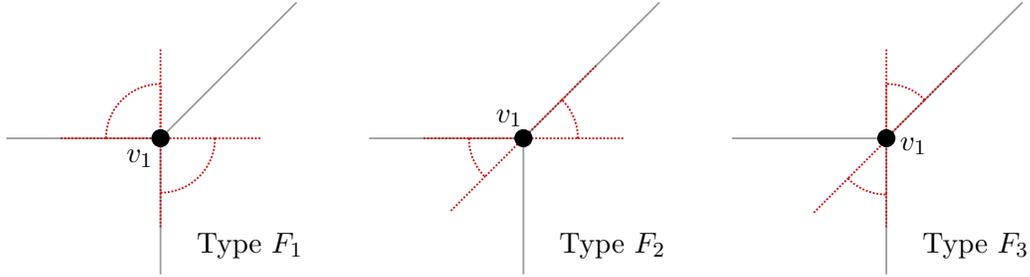
\caption{A vertex $v_1$ with its cones in which $b_1+v_1$ and $e+v_1$ can lie. From left to right: A vertex $v_1$ of type $F_1,F_2$ and $F_3$.}
\label{Figure_7}
\end{figure}

The other way round, given a vertex $v_1\in\tilde{\Gamma}$ and its type $F_i$, we can estimate the positions of $b_1+v_1$ and $e+v_1$. See Figure \ref{Figure_7} for the following: If $v_1$ is of type $F_i$, then $b_1+v_1$ and $e+v_1$ need to lie in the red cones depicted in Figure \ref{Figure_7} in such a way that $b_1+v_1$ and $e+v_1$ lie in opposite cones.
\end{classification}

\begin{remark}
If there is some maximal chain $v_1\geq \dots \geq v_n$ in $\tilde{\Gamma}$ with respect to $\Omega(H)$ such that $b_1+v_1\in H_{v_1}$ and $v_1$ is of type $F_i$, then $v_j$ is also of type $F_i$ for $j=2,\dots,n$.
\end{remark}

\begin{proof}
We use Notation \ref{notation:direction_of_movement}.
By induction, is is sufficient to show the statement for $v_1\geq v_2$. Let $e_1$ be the edge adjacent to $v_1,v_2$. Let $F_i$ be the type of $v_1$ such that $\sigma_{e_1}+v_1$ and $\sigma_{b_1}+v_1$ are its two opposing cones, where $e_1+v_1\in\sigma_{e_1}+v_1$ and $b_1+v_1\in\sigma_{b_1}+v_1$. Hence $-e_1+v_2\in\sigma_{b_1}+v_2$. By Observation \ref{observation:half-planes}, we obtain $b_2+v_2\in\sigma_{e_1}+v_2$.
\end{proof}

%

\begin{lemma}\label{lemma:about_chains_strong_vorarbeit}
We use Notation \ref{notation:direction_of_movement}.
Let $\tilde{\Gamma}$ be a tropical curve in $\mathbb{R}^2$ as in Construction \ref{constr:tilde_Gamma} that allows an unbounded $1$-dimensional movement. Let $v_1$ be a vertex of the movable component of $\tilde{\Gamma}$. Let $H$ be an open half-plane. Let $v_1\geq\dots\geq v_n$ be a maximal chain with respect to $\Omega(H)$ such that $n>1$ and $b_1+v_1\in H_{v_1}$. If $b_n$ is of non-standard direction, then $v_n$ is adjacent to two ends of $\tilde{\Gamma}$ of different standard directions. If $b_n$ is of standard direction, then $v_n$ is adjacent to one end of $\tilde{\Gamma}$ of standard direction parallel to $b_n$.
\end{lemma}

\begin{proof}
Assume that $v_n$ is of type $F_i$ for an $i=1,2,3$ and that $b_n$ is of non-standard direction. Thus, by Classification \ref{classification:types_of_vertices_2}, $b_n+v_n$ lies in the interior of one of the dashed red cones of Figure \ref{Figure_7} and all bounded edges adjacent to $v_n$ lie in the opposite cone.
Therefore, by the balancing condition, $v_n$ needs to be adjacent to at least two ends of different standard directions.

Next, assume that $b_n$ is of standard direction. Hence $b_n+v_n$ appears in the boundary of two of the red cones $\sigma_1,\sigma_2$ of Classification \ref{classification:types_of_vertices_2}. Therefore all edges which are no ends adjacent to $v_n\in\tilde{\Gamma}$ are in the union $\sigma_1'\cup \sigma_2'$ of the opposite cones $\sigma_j'$ of $\sigma_j$ for $j=1,2$. Therefore balancing guarantees that there is an end adjacent to $v_n\in\tilde{\Gamma}$ which is parallel to $b_n$.
\end{proof}

The following Lemma generalizes Lemma \ref{lemma:about_chains_weak} from $1$-ray half-planes to arbitrary half-planes.

\begin{lemma}[About maximal chains, strong version]\label{lemma:about_chains_strong}
Let $\tilde{\Gamma}$ be a tropical curve in $\mathbb{R}^2$ as in Construction \ref{constr:tilde_Gamma} that allows an unbounded $1$-dimensional movement. Let $v_1$ be a vertex of the movable component of $\tilde{\Gamma}$. If there is an open half-plane $H$ such that $v_1\geq\dots\geq v_n$ is a maximal chain starting at $v_1$ with respect to $\Omega(H)$ such that $n>1$ and $b_1+v_1\in H_{v_1}$, then $v_n$ is a $3$-valent type (I) vertex.
\end{lemma}

\begin{proof}
We use Notation \ref{notation:direction_of_movement}, assume that $\val(v_n)>3$, that $v_n$ is of type $F_i$ for an $i=1,2,3$ and that $b_n$ is of non-standard direction. By Lemma \ref{lemma:about_chains_strong_vorarbeit}, $v_n$ needs to be adjacent to at least two ends $E_1,E_2$ of different standard directions.
By Corollary \ref{cor:CR_pfade_ueber_alle_edges_an_vertex}, we can reach a type (IIIb) vertex via each of the edges $E_1,E_2$ in $\Gamma$. The direction of movement of such a type (IIIb) vertex cannot be parallel to the end of standard direction it is connected to, otherwise we would have a $2$-dimensional movement. Recall that type (IIIb) vertices can only move in standard direction since their contracted ends satisfy multi line conditions.
See Figure \ref{Figure_8} for the following: If $i=1$, i.e. $v_n$ is of type $F_1$, we consider the cone in which $b_n+v_n$ lies and go through all different directions of movements of the type (IIIb) vertices. In each case we obtain a contradiction to your unbounded movement.

We still get a contradiction if $b_n+v_n$ would lie in the other red cone of Figure \ref{Figure_8}. More generally, the same arguments and conclusion of the case $i=1$ are true for $i=2,3$ and lead to contradictions as well.

\begin{figure}[H]
\centering
\def\svgwidth{440pt}
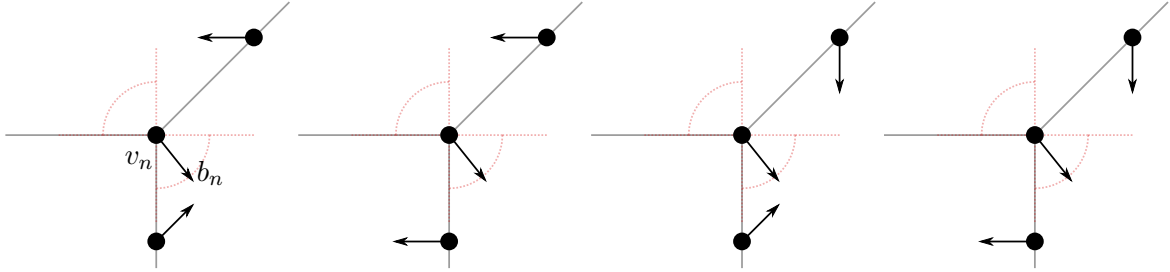
\caption{A vertex $v_n$ of type $F_1$ connected to two type (IIIb) vertices which move along the directions of the arrows.}
\label{Figure_8}
\end{figure}

Next, we assume that $b_n$ is of standard direction. By Lemma \ref{lemma:about_chains_strong_vorarbeit}, there is an end $E_1$ adjacent to $v_n\in\tilde{\Gamma}$ which is parallel to $b_n$. Since we assumed that $\val(v_n)>3$, there must, again,  be a type (IIIb) vertex adjacent to $v_n$ via $E_1$. Notice that this vertex can only move unboundedly in the direction of $b_n$, which is a contradiction because our movement is only $1$-dimensional.

In total, $v_n$ can only be a type (I) vertex that is $3$-valent.
\end{proof}

\begin{corollary}\label{corollary:no_chains_strong}
Let $v_1,b_1$ and $H$ be an open half-plane as in Lemma \ref{lemma:about_chains_strong}. Then there is no chain $v_1\geq \dots \geq v_n$ with $n>1$ and $b_1+v_1\in H_{v_1}$ in the movable component of $\tilde{\Gamma}$.
\end{corollary}

\begin{proof}
We use Notation \ref{notation:direction_of_movement} and assume that there is a maximal chain $v_1\geq \dots \geq v_n$ starting at $v_1$. Hence $v_n$ must be a $3$-valent type (I) vertex by Lemma \ref{lemma:about_chains_strong}.
By Lemma \ref{lemma:about_chains_strong_vorarbeit}, there is an end $E$ of $\tilde{\Gamma}$ adjacent to $v_n$
Moreover, denote the direction vector at $v_n$ of the edge that connects $v_n$ to a fixed component by $f$. Therefore the direction of movement of $v_n$, denoted by $b_n$, is given by $-f$ since $v_n$ moves unboundedly, i.e. it moves away from the fixed component it is adjacent to.
We distinguish all cases of Classification \ref{classification:types_of_vertices_2} for $v_n$. So let the type of the vertex $v_n$ be $F_i$ for an $i=1,2,3$ (see Figure \ref{Figure_7}). Since $b_n=-f$, the edges $e_{n-1}$ and $f$ adjacent to $v_n$ lie in the same cone. Then there exists no end $E$ such that $v_n$ is balanced (for each possible end $E$ we find a half-plane $P$ such that $E+v_n,f+v_n,-e_{n-1}+v_n\in P_{v_n}$) which is a contradiction.
\end{proof}

The following proof builds on ideas of Proposition 5.1 in \cite{KontsevichPaper}.

\begin{proof}[Proof of Proposition \ref{prop:contracted_bounded_edge}]
Consider the $1$-dimensional cycle
\begin{align*}
Y=\prod_{k\in\underline{\kappa}}\ev_k^*(L_k)\cdot \prod_{i\in\underline{n}} \ev_i^*\left( p_i\right)\cdot \prod_{j=1}^{l-1} \ft_{\lambda_j}^*\left( 0\right) \cdot \mathcal{M}_{0,m}\left(\mathbb{R}^2,d\right)
\end{align*}
from Definition \ref{def:movable_component}. We need to show that $\lbrace \ft_{\lambda'_l}(C)\mid C\in Y \textrm{ has no contracted bounded edge}\rbrace$ is bounded in $\mathcal{M}_{0,4}$. If it is unbounded, then there is a curve $C$ coming from a stable map in $Y$ without a contracted bounded edge which allows an unbounded movement. Hence the movable component of $C$ has exactly one vertex $v$ by Corollary \ref{corollary:no_chains_strong} which is not of type (IIIa) or (IIIb) as in Classification \ref{classification:types_of_vertices}. Notice that $C$ has at least one fixed component as well since we assume that there is at least one point condition that $C$ satisfies.

\begin{figure}[H]
\centering
\def\svgwidth{440pt}
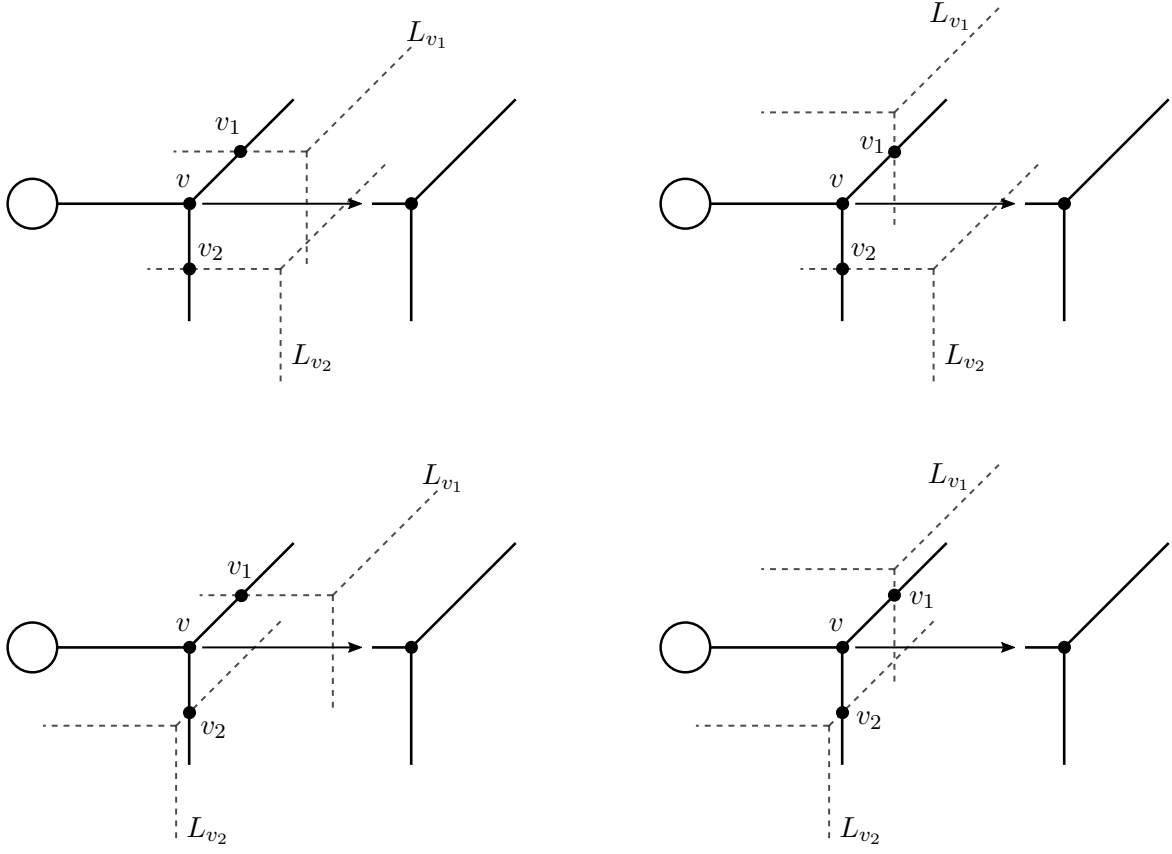
\caption{The movable vertex $v$ and its movement away from the fixed component.}
\label{Figure_18}
\end{figure}

We distinguish different cases for $v$.
\begin{itemize}
\item[(1)]
Assume that $\val(v)=3$ and that $v$ is adjacent to two edges $E_1,E_2$ which are parallel to two ends of different direction. The edges $E_1,E_2$ lead to other vertices in the movable component moving $v$ varies $\ft_{\lambda'_l}(C)$ and Corollary \ref{cor:CR_pfade_ueber_alle_edges_an_vertex} applies. There are $3$ cases (choose $2$ different directions for $E_1,E_2$ from the $3$ standard directions) we need to distinguish. Moving $v$ unboundedly, we obtain an end adjacent to $v$. More precisely, Figure \ref{Figure_18} shows one of the $3$ case where the directions are $(1,1)$ and $(0,-1)$ (the other two cases are analogous). Hence moving $v$ further in its direction of movement eventually produces a combinatorial type that does not allow $\ft_{\lambda'_l}(C)$ to become larger as $v$ is moved.
\item[(2)]
Assume that $\val(v)=3$ and that all edges adjacent to $v$ are parallel. Since all ends of $C$ are of weight $1$, the two edges $E_1,E_2$ adjacent to $v$, which lead to other vertices in the movable component, are on the same side of $v$. Therefore moving $v$ as before (analogous to Figure \ref{Figure_18} but with $v_1,v_2$ lying on parallel ends) does not make the coordinate $\ft_{\lambda'_l}(C)$ larger.
\item[(3)]
Assume that $\val(v)>3$, then there are edges $E_1,E_2$ adjacent to $v$ (by Corollary \ref{cor:CR_pfade_ueber_alle_edges_an_vertex}) which connect $v$ to vertices $v_1,v_2$ of the movable component that satisfy line conditions $L_{v_1},L_{v_2}$. The same movement as in the case of $\val(v)=3$ yields a combinatorial type where there is an end adjacent to $v$ which contradicts Corollary \ref{cor:CR_pfade_ueber_alle_edges_an_vertex} since $\val(v)>3$, see again Figure \ref{Figure_18}.
\end{itemize}
In total, choosing a large value for $|\lambda'_l|$ implies that only curves with a contracted bounded edge can contribute to $N_{d}\left(p_{\underline{n}},L_{\underline{\kappa}}, \lambda_{[l-1]},\lambda'_l \right)$. Moreover, there is exactly one contracted bounded edge. Otherwise a stable map $C$ contributing to $N_{d}\left(p_{\underline{n}},L_{\underline{\kappa}}, \lambda_{[l-1]},\lambda'_l \right)$ would give rise to a $1$-dimensional family of stable maps contributing to $N_{d}\left(p_{\underline{n}},L_{\underline{\kappa}}, \lambda_{[l-1]},\lambda'_l \right)$ which is a contradiction.
\end{proof}

Notice that in Proposition \ref{prop:contracted_bounded_edge} we assumed that $n\geq 1$, i.e. that there is at least one point condition. However, even without point conditions we can still assume that there is a contracted bounded edge, see Proposition \ref{prop:contracted_bounded_edge_no_point_conditions}.

\begin{lemma}\label{lemma:initial_values_1}
Let $C$ be a tropical stable map that contributes to $N_{d}\left(L_{\underline{\kappa}}, \lambda_{[l]} \right)$. Then there is a vertex $v$ of $C$ which is adjacent to two contracted ends $e_1,e_2$ such that $e_1$ satisfies a multi line condition $L_a$ and $e_2$ satisfies a multi line condition $L_b$, respectively.
\end{lemma}
\begin{proof}
Assume that each vertex of $C$ is at most adjacent to one contracted end that satisfies a multi line condition. Hence each vertex of the tropical curve associated to $C$ allows a $1$-dimensional movement since its movement is only restricted by at most one multi line condition (we have no point conditions). Thus $C$ give rise to a $1$-dimensional family which is a contradiction.
\end{proof}

\begin{lemma}\label{lemma:initial_values_2}
Let $v$ be the vertex adjacent to $e_1,e_2$ from Lemma \ref{lemma:initial_values_1}. Then $\val(v)>3$ and there is a degenerated cross-ratio $\lambda\in\lambda_{[l]}$ such that $\lambda=\lbrace e_1, e_2, \beta_3,\beta_4\rbrace $.
\end{lemma}

\begin{proof}
We use the notation from Lemma \ref{lemma:initial_values_1}. If $\val(v)=3$, then, by Lemma \ref{lemma:initial_values_1}, there is a contracted bounded edge adjacent to $v$. Hence $C$ cannot be fixed by the set of given conditions which is a contradiction. Thus $\val(v)>3$.

By Corollary \ref{cor:CR_pfade_ueber_alle_edges_an_vertex} there is a cross-ratio $\lambda$ as desired or there are cross-ratios $\lambda_1=\lbrace e_1,\dots \rbrace$ and $\lambda_2=\lbrace e_2,\dots \rbrace$ such that $e_2\notin \lambda_1$ and $e_1\notin \lambda_2$. Assume that there is no cross-ratio $\lambda$ as desired. Then $v$ can be resolved by adding a contracted bounded edge $e$ to $C$ that is adjacent to $v$ and a new $3$-valent vertex $v'$ which is adjacent to $e_1,e_2$. Notice that this resolution of $v$ is compatible with $\lambda_1,\lambda_2$ but gives rise to a $1$-dimensional family of tropical stable maps satisfying $L_{\underline{\kappa}}, \lambda_{[l]}$ which is a contradiction.
\end{proof}

\begin{proposition}\label{prop:contracted_bounded_edge_no_point_conditions}
We use notation from Lemma \ref{lemma:initial_values_1} and Lemma \ref{lemma:initial_values_2} and assume without loss of generality that $e_1,e_2$ are entries of the cross-ratio $\lambda_l$. Let $\lambda'_l$ be a non-degenerated cross-ratio that degenerates to $\lambda_l$, where $e_1,e_2$ are grouped together. Then every tropical stable map $C'$ that contributes to $N_{d}\left(L_{\underline{\kappa}}, \lambda_{[l-1]},\lambda_l' \right)$ arises from a tropical stable map $C$ that contributes to $N_{d}\left(L_{\underline{\kappa}}, \lambda_{[l]} \right)$ by adding a contracted bounded edge $e$ to $C$ that is adjacent to $v$ and a new vertex $v'$ which is in turn adjacent to $e_1,e_2$.
\end{proposition}

\begin{proof}
Let $C$ be a tropical stable map that contributes to $N_{d}\left(L_{\underline{\kappa}}, \lambda_{[l]} \right)$ and let $v$ the vertex from Lemma \ref{lemma:initial_values_1} at which $\lambda_l$ is satisfied. Assume that the edge $e'$ we add by resolving $v$ according to $\lambda_l'$ is not contracted and denote the tropical stable map obtained this way ba $C''$. Denote the vertex adjacent to $e'$ and $e_1,e_2$ by $\tilde{v}$.
Consider $C''$ as a point in the cycle that arises from dropping the cross-ratio condition $\lambda_l'$ (cf. Definition \ref{def:movable_component}). Then $C''$ is in the boundary of a $2$-dimensional cell of the same cycle that arises from $C''$ by adding a contracted bounded edge $e$ to $C''$ that separates $\tilde{v}$ from $e_1,e_2$. Hence there is a $2$-dimensional cell inside a $1$-dimensional cycle, which is a contradiction. 

Each tropical stable map $C$ contributes to $N_{d}\left(L_{\underline{\kappa}}, \lambda_{[l]} \right)$ yields a contribution to $N_{d}\left(L_{\underline{\kappa}}, \lambda_{[l-1]},\lambda_l' \right)$ if the vertex $v$ at which $\lambda_l$ is satisfied is resolved according to $\lambda_l'$ and each resolution of $v$ according to $\lambda_l'$ produces a contracted bounded edge $e$. Hence Remark \ref{remark:numbers_independent_of_positions_and_lengths} and the description of $\mult(C)$ via resolutions of vertices (see also \cite{CR1}) guarantees that there cannot be more stable maps $C'$ contributing to $N_{d}\left(L_{\underline{\kappa}}, \lambda_{[l-1]},\lambda_l' \right)$ than the ones obtained from adding a contracted bounded edge $e$ to tropical stable maps $C$.
\end{proof}

\subsection*{Behavior of cut contracted bounded edges} 
After we identified a contracted bounded edge $e$ in Propositions \ref{prop:contracted_bounded_edge}, \ref{prop:contracted_bounded_edge_no_point_conditions}, we can cut this edge which yields a split of the original tropical stable map into two new ones.
The aim of this subsection is to prove Corollary \ref{corollary:1/1-edge_cut_standard_direction}, in which the behavior of the two new ends that arise from cutting $e$ is described.

\begin{construction}[Cutting the contracted bounded edge]\label{constr:cutting_contracted_bounded_edge}
Let $C$ be a tropical stable map that contributes to $N_{d}\left(p_{\underline{n}},L_{\underline{\kappa}}, \lambda_{[l-1]},\lambda'_l \right)$, where $\lambda'_l$ is a non-degenerated tropical cross-ratio such that $|\lambda'_l|$ is large. Assume that $C$ has a contracted bounded edge $e$.

If we cut $e$, we obtain two tropical stabel maps $C_1$ and $C_2$ with contracted ends $e_1$ and $e_2$ that come from $e$. By abuse of notation, the label of $e_i$ is also $e_i$ for $i=1,2$. We usually denote the vertices adjacent to the ends $e_1,e_2$ by $v_1,v_2$.
Notice that $C_i$ is of degree $d_i$ for $i=1,2$ such that $d_1+d_2=d$ since $C$ is balanced and of degree $d$.

If a contracted bounded edge $e$ is cut, the cross-ratios can be \textit{adapted} the following way: If $\lambda_j$ is a degenerated cross-ratio that is satisfied at some vertex $v\in C_i$ for $i=1,2$, then, by the path criterion (Remark \ref{remark:path_criterion}), either all entries of $\lambda_j$ are labels of contracted ends of $C_i$ or $3$ entries of $\lambda_j$ are labels of contracted ends of $C_i$ and one entry $\beta$ is a label of a contracted end of $C_t$ for $t\neq i$. In the first case, we do not change $\lambda_j$ and in the latter case, we replace the entry $\beta$ of $\lambda_j$ by $e_i$. We denote a degenerated cross-ratio that we adapted to $e_i$ by $\lambda_j^{\to e_i}$.

Each $C_i$ of degree $d_i$ for $i=1,2$ satisfies point conditions $p_{\underline{n_i}}$, multi line conditions $L_{\underline{\kappa_i}}$ and cross-ratio conditions $\lambda^{\to e_i}_{\underline{l_i}}$ such that $\underline{n_1}\cupdot\underline{n_2}=\underline{n}$, $\underline{\kappa_1}\cupdot\underline{\kappa_2}= \underline{\kappa}$ and $\underline{l_1}\cupdot\underline{l_2}=[l-1]$, where we adapted all cross-ratios to the cut edge $e$. We say that $C$ \textit{splits} into the two tropical stable maps $C_1$ and $C_2$ and the \textit{splitting type} of $C$ is $(d_1,\underline{n_1},\underline{\kappa_1},\underline{l_1},\underline{f_1}\mid d_2,\underline{n_2},\underline{\kappa_2},\underline{l_2},\underline{f_2})$, where $\underline{f_1}\cupdot \underline{f_2}=\underline{f}$ is a partition of the ends of $C$ that satisfy no point or multi line condition as in Definition \ref{def:general_pos}.
\end{construction}

\begin{definition}[$1/1$ and $2/0$ splits]\label{def:splits}
Let $d$ be a degree, let $p_{\underline{n}},L_{\underline{\kappa}}, \lambda_{[l-1]}$ be given conditions and let $\underline{f}$ be labels of contracted ends that satisfy no conditions as in Definition \ref{def:cycle_Z_d}.
We refer to $(d_1,\underline{n_1},\underline{\kappa_1},\underline{l_1},\underline{f_1}\mid d_2,\underline{n_2},\underline{\kappa_2},\underline{l_2},\underline{f_2})$ as a \textit{split} (of conditions) if $d_1+d_2=d$, $\underline{n_1}\cupdot\underline{n_2}=\underline{n}$, $\underline{\kappa_1}\cupdot\underline{\kappa_2}= \underline{\kappa}$, $\underline{l_1}\cupdot\underline{l_2}=[l-1]$, $\underline{f_1}\cupdot \underline{f_2}=\underline{f}$ holds and each cross-ratio in $\lambda_{\underline{l_i}}$ has at least $3$ of its entries in $\underline{n_i}\cup\underline{\kappa_i}\cup\underline{f_i}$. If we write $\lambda^{\to e_i}_{\underline{l_i}}$, we mean that each entry of each cross-ratio in $\lambda_{\underline{l_i}}$ that is not in $\underline{n_i}\cup\underline{\kappa_i}\cup\underline{f_i}$ is replaced by the label $e_i$.
Such a split is called a $1/1$ \textit{split} if
\begin{align}\label{eq:equation_for_1/1_split}
3d_i=\#\underline{n_i}+\#\underline{l_i}-\#\underline{f_i}+1
\end{align}
holds for $i=1,2$. If
\begin{align}\label{eq:equation_for_2/0_split}
3d_i=\#\underline{n_i}+\#\underline{l_i}-\#\underline{f_i} \textrm{ and } 3d_t=\#\underline{n_t}+\#\underline{l_t}-\#\underline{f_t}+2
\end{align}
holds for $i=1,2$ with $t\neq i$ for some choice of $i,t\in\lbrace 1,2 \rbrace$, then we refer to $(d_1,\underline{n_1},\underline{\kappa_1},\underline{l_1},\underline{f_1}\mid d_2,\underline{n_2},\underline{\kappa_2},\underline{l_2},\underline{f_2})$ as a $2/0$ \textit{split}.
\end{definition}

\begin{definition}[$1/1$ and $2/0$ edges]\label{def:1/1_and_2/0_edges}
Let $(d_1,\underline{n_1},\underline{\kappa_1},\underline{l_1},\underline{f_1}\mid d_2,\underline{n_2},\underline{\kappa_2},\underline{l_2},\underline{f_2})$ be a split of conditions as in Definition \ref{def:splits}. Define for the (adapted) conditions $p_{\underline{n_i}}, L_{\underline{\kappa_i}}, \lambda^{\to e_i}_{\underline{l_i}}$ and for $i=1,2$ the cycles
\begin{align*}
Y_{i}:=\ev_{e_i,*}\left( \prod_{k\underline{\kappa_i}}\ev_k^*(L_k)\cdot \prod_{t\in\underline{n_i}} \ev_t^*\left( p_t\right)\cdot \prod_{j\in \underline{l_i}} \ft_{\lambda_j^{\to e_i}}^*\left( 0\right) \cdot \mathcal{M}_{0,m_i}\left(\mathbb{R}^2,d_i\right) \right)\subset\mathbb{R}^2,
\end{align*}
where $m_i:=\#\underline{n_i}+\#\underline{\kappa_i}+\#\underline{f_i}$.
Notice that $(d_1,\underline{n_1},\underline{\kappa_1},\underline{l_1},\underline{f_1}\mid d_2,\underline{n_2},\underline{\kappa_2},\underline{l_2},\underline{f_2})$ is a $1/1$ split if and only if both $Y_i$ are $1$-dimensional. It is a $2/0$ split if and only if $Y_i$ is $0$-dimensional and $Y_t$ is $2$-dimensional (see \eqref{eq:equation_for_2/0_split} in Definition \ref{def:splits}).

Let $C$ be a tropical stable map with a contracted bounded edge $e$ such that $C$ is of splitting type $(d_1,\underline{n_1},\underline{\kappa_1},\underline{l_1},\underline{f_1}\mid d_2,\underline{n_2},\underline{\kappa_2},\underline{l_2},\underline{f_2})$. Then $m_i$ is the number of contracted ends of $C_i$ and the cycle $Y_i$ is the condition $C_i$ imposes on $C_t$ for $t\neq i$ via $e$.
For example, if $Y_1$ is $0$-dimensional, then the position of $v_2$ is completely determined by $Y_1$ since $v_2$ is connected to $v_1$ via $e$ in $C$ and $C$ is fixed by the given conditions $p_{\underline{n}}, L_{\underline{\kappa}},  \lambda_{[l-1]}, \lambda_l$. Since all given conditions are in general position, the dimension of $Y_2$ is $2$ in this case, i.e. $v_2$ cannot impose a condition via $e$ to $v_1$. In general, we have two cases for $C$:
\begin{itemize}
\item[(1)]
One of the cycles $Y_i$ is $0$-dimensional and the other one is $2$-dimensional. We then refer to $e$ as a \textit{$2/0$ edge}.
\item[(2)]
Both of the cycles $Y_i$ are $1$-dimensional. We then refer to $e$ as a \textit{$1/1$ edge}.
\end{itemize}
Which case occurs depends only on $d_i,\#\underline{n_i},\#\underline{\kappa_i},\#\underline{l_i},\#\underline{f_i}$ for $i=1,2$.
\end{definition}

\begin{example}\label{ex:2/0_split}
An example for a $1/1$ split is provided below, see Example \ref{ex:1/1_split}. An example for a $2/0$ split is the following:
Let $C$ be a degree $2$ tropical stable map that satisfies point conditions $p_{[2]}$, multi line conditions $L_{[4]}$, degenerated cross-ratios $\lambda_1=\lbrace p_1,L_1,L_2,L_3 \rbrace$, $\lambda_2=\lbrace p_1,p_2,L_1,L_2 \rbrace$ and a non-degenerated cross-ratio $\lambda'_3=(p_1L_1|p_2L_4)$ whose length is large enough such that $C$ has a contracted bounded edge $e$. Construction \ref{constr:cutting_contracted_bounded_edge} yields a split of $C$ into $C_1$ and $C_2$, where the vertices adjacent to the split edge $e$ are denoted by $v_i\in C_i$ for $i=1,2$. Figure \ref{Figure_24} shows $C_1$ and $C_2$, where we shifted $C_2$ in order to get a better picture (in fact $v_1$ and $v_2$ are the same point in $\mathbb{R}^2$).

\begin{figure}[H]
\centering
\def\svgwidth{350pt}
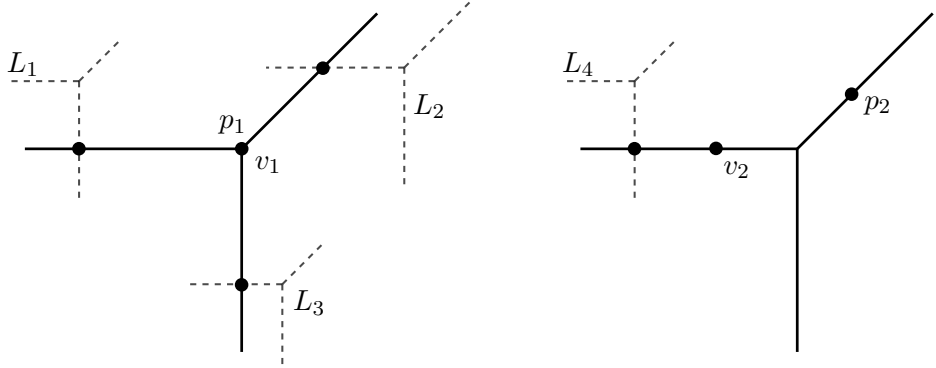
\caption{The curve $C_1$ satisfying $p_{1},L_{[3]},\lambda_{[2]}$ is shown on the left, the curve $C_2$ satisfying $p_2,L_4$ is shown on the right. Notice that the length of $e$ in $C$ is given by $\lambda'_3$, i.e. $C$ is fixed by the given conditions.}
\label{Figure_24}
\end{figure}

\end{example}

\begin{remark}\label{remark:split_of_conditions}
Fix a degree $d$, point conditions $p_{\underline{n}}$, multi line conditions $L_{\underline{\kappa}}$ and cross-ratio conditions $\lambda_{[l-1]}$. Let $(d_1,\underline{n_1},\underline{\kappa_1},\underline{l_1},\underline{f_1}\mid d_2,\underline{n_2},\underline{\kappa_2},\underline{l_2},\underline{f_2})$ denote a split of these conditions.
Consider degree $d_i$ tropical stable maps $C_i$ for $i=1,2$ with $\#\underline{n_i}+\#\underline{\kappa_i}+\#\underline{f_i}+1$ contracted ends that satisfy the point conditions $p_{\underline{n_i}}$, the multi line conditions $L_{\underline{\kappa_i}}$ and the cross-ratio conditions $\lambda^{\to e_i}_{\underline{l_i}}$.
The cycles $Y_i$ for $i=1,2$ tell us how to glue the end $e_1$ of $C_1$ to the end $e_2$ of $C_2$ to form a contracted bounded edge $e$ such that the new tropical stable map $C$ satisfies all given conditions.

If $Y_1$ is $0$-dimensional and $p_{e_2}$ is a point in $Y_1$, then considering tropical stable maps $C_2$ that satisfy $p_{\underline{n_2}},L_{\underline{\kappa_2}},\lambda^{\to e_2}_{\underline{l_2}}$ and that satisfy $p_{e_2}$ with the end $e_2$ allows us to glue $C_1$ to $C_2$, where the contracted bounded edge is contracted to $p_{e_2}\in\mathbb{R}^2$.

If both $Y_i$ are $1$-dimensional, then we can consider tropical stable maps $C_2$ that satisfy $p_{\underline{n_2}},L_{\underline{\kappa_2}},\lambda^{\to e_2}_{\underline{l_2}}$ and $Y_1$. Since $\ev_{e_2}(C_2)\in Y_2$, i.e. $C_2$ satisfies $Y_2$ by definition, the position of the contracted end $e_2$ of $C_2$ in $\mathbb{R}^2$ is a point $p$ contributing to the $0$-dimensional cycle $Y_1 \cdot Y_2$. On the other hand, there is a tropical stable map $C_1$ that satisfies $p_{\underline{n_2}},L_{\underline{\kappa_2}},\lambda^{\to e_1}_{\underline{l_2}}$ and $Y_2$ such that its end $e_1$ is contracted to $p$. Thus $e_1$ of $C_1$ and $e_2$ of $C_2$ can be glued to form a bounded edge $e$ that is contracted to $p$.
\end{remark}

\begin{corollary}[of Proposition \ref{prop:contracted_bounded_edge}]\label{corollary:1/1-edge_cut_standard_direction}
If $C$ is a tropical stable map as in Proposition \ref{prop:contracted_bounded_edge} whose contracted bounded edge is a $1/1$ edge, then the $1$-dimensional cycles $Y_i$ from Definition \ref{def:1/1_and_2/0_edges} have ends of primitive directions $(1,1),(-1,0)$ and $(0,-1)\in\mathbb{R}^2$ only. In other words, the $1$-dimensional conditions that a contracted bounded $1/1$ edge passes from one vertex to the other has ends of standard directions.
\end{corollary}

\begin{proof}
Proposition \ref{prop:contracted_bounded_edge_no_point_conditions} implies that each contracted bounded edge that appears in the no-point-conditions case is a $2/0$ edge. Hence we may assume that at least one point condition is given.

Let $\Gamma$ be a tropical curve associated to a tropical stable map in $Y_i$ whose movement is unbounded, i.e. that gives rise to an end of $Y_i$. Corollary \ref{corollary:no_chains_strong} yields that the movable component of $\Gamma$ consists of exactly one vertex $v_i$ of type (I) or (II). Thus $v_i$ is of type (I) since we assumed that there is at least one point condition.
If there is a cross-ratio $\lambda_j\in\lambda_{[l-1]}$ such that $\lambda^{\to e_i}_j$ is satisfied at $v_i$, i.e. $\lambda^{\to e_i}_j\in\lambda_v$, then Corollary \ref{cor:CR_pfade_ueber_alle_edges_an_vertex} guarantees that $v_i$ is not adjacent to unbounded edges. This yields a contradiction when $v_i$ moves unboundedly as the proof of Proposition \ref{prop:contracted_bounded_edge} shows. Hence $v_i$ is a $3$-valent type (I) vertex which is adjacent to $e_i$ and an end $E$ of $\Gamma$. Therefore, $v_i$ moves parallel to $E$.
\end{proof}

\begin{corollary}\label{corollary:1/1-edge_cut_v_i_3_valent_adjacent_to_end}
We use Notation from Construction \ref{constr:cutting_contracted_bounded_edge}, i.e. we denote the vertex adjacent to the end $e_i$ of $C_i$ by $v_i$. Under the same assumptions of Corollary \ref{corollary:1/1-edge_cut_standard_direction}, it follows that $v_i$ is $3$-valent and adjacent to an end of $C_i$ for $i=1,2$.
\end{corollary}

\begin{proof}
This follows immediately from the proof of Corollary \ref{corollary:1/1-edge_cut_standard_direction}.
\end{proof}

\section{Multiplicities of split curves}
This section answers the question of how multiplicities behave under splitting a tropical stable map $C$ into $C_1,C_2$. Note that the multiplicity of $C$ does not have to be equal to $\mult(C_1)\cdot \mult(C_2)$. We have to deal with this problem later.

\begin{definition}[Degenerated tropical lines]\label{def:degenerated_trop_lines}
The tropical intersections $L_{10}:=\max_{(x,y)\in\mathbb{R}^2}(x,0)\cdot \mathbb{R}^2$, $L_{01}:=\max_{(x,y)\in\mathbb{R}^2}(y,0)\cdot \mathbb{R}^2$ and $L_{1\text{-}1}:=\max_{(x,y)\in\mathbb{R}^2}(x,-y)\cdot \mathbb{R}^2$ and any translations thereof are called \textit{degenerated tropical lines}.
\end{definition}

\begin{figure}[H]
\centering
\def\svgwidth{350pt}
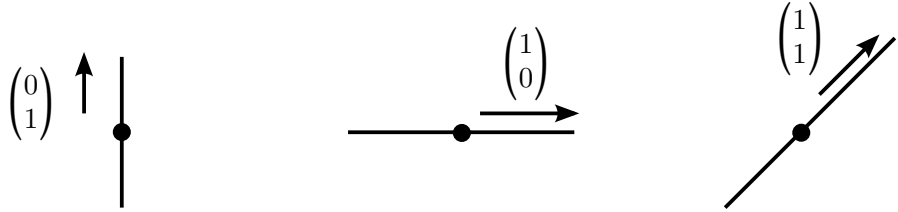
\caption{Degenerated tropical lines (from left to right) $L_{10},L_{01}$ and $L_{1\text{-}1}$ in $\mathbb{R}^2$ with ends of weight one.}
\label{Figure14}
\end{figure}

\begin{notation}[Replacing $1/1$ edge conditions]\label{notation:replacing_1/1-edge_conditions}
Let $C$ be a tropical stable map that contributes to $N_{d}\left(p_{\underline{n}},L_{\underline{\kappa}}, \lambda_{[l-1]},\lambda'_l \right)$ such that $C$ has a contracted bounded edge $e$ that is a $1/1$ edge. Split $e$ as in Construction \ref{constr:cutting_contracted_bounded_edge} to obtain $C_1,C_2$ and let $Y_t$ denote the $1$-dimensional condition $C_i$ satisfies for $i\neq t$ as in Definition \ref{def:1/1_and_2/0_edges}.
Let $v_i$ be the vertex of $C_i$ that is adjacent to $e_i$ ($e_i$ is the contracted end of $C_i$ that came from cutting $e$) which satisfies $Y_t$. Let $st\in\lbrace 01,10,1\text{-}1 \rbrace$ and let $L_{st}$ be a degenerated line as in Definition \ref{def:degenerated_trop_lines} such that its vertex is translated to $v_i$. Let $C_{i,st}$ denote the tropical curve that equals $C_i$, but where we replaced the $Y_t$ conditions with $L_{st}$, i.e. $C_{i,st}$ satisfies $L_{st}$ instead of $Y_t$.

Notice that only the multiplicities of $C_i$ and $C_{i,st}$ may differ. In particular, the multiplicity of $C_{i,st}$ may be zero, whereas the multiplicity of $C_i$ can be nonzero.
\end{notation}

\begin{example}\label{ex:1/1_split}
Let $C$ be a degree $3$ tropical stable map that satisfies point conditions $p_{[5]}$, multi line conditions $L_{[3]}$, degenerated cross-ratios $\lambda_1=\lbrace p_1,p_2,p_5,L_1 \rbrace$, $\lambda_2=\lbrace p_1,p_5,L_2,L_3 \rbrace$ and a non-degenerated cross-ratio $\lambda'_3=(p_1p_2|L_2L_3)$ whose length is large enough such that $C$ has a contracted bounded edge $e$. Construction \ref{constr:cutting_contracted_bounded_edge} yields a split of $C$ into $C_1$ and $C_2$, where the vertices adjacent to the split edge $e$ are denoted by $v_i\in C_i$ for $i=1,2$. Figure \ref{Figure_21} shows $C_1$ and $C_2$, where we shifted $C_2$ in order to get a better picture (in fact $v_1$ and $v_2$ are the same point in $\mathbb{R}^2$ as in Example \ref{ex:2/0_split}). 

\begin{figure}[H]
\centering
\def\svgwidth{350pt}
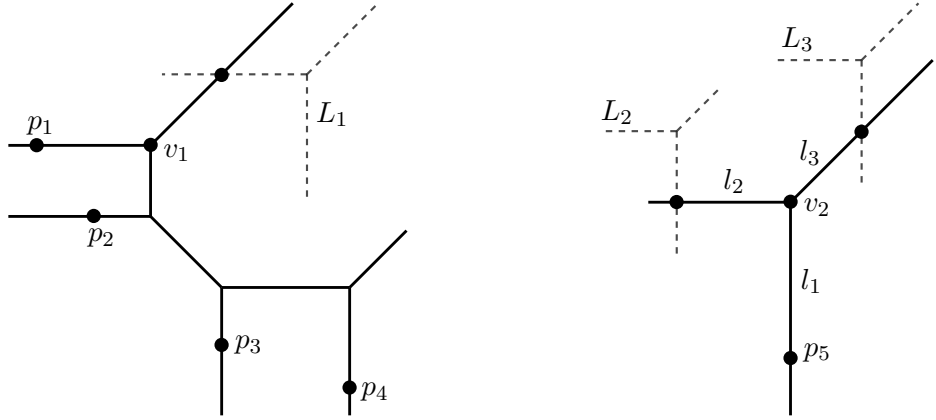
\caption{The curve $C_1$ satisfying $p_{[4]},L_1,\lambda_1$ is shown on the left, the curve $C_2$ satisfying $p_5,L_2,L_3,\lambda_2$ is shown on the right. Notice that the length of $e$ in $C$ is given by $\lambda'_3$, i.e. $C$ is fixed by the given conditions.}
\label{Figure_21}
\end{figure}

Notice that $e$ is a $1/1$ edge, so we use Notation \ref{notation:replacing_1/1-edge_conditions} to replace conditions. For example, $C_{2,10}$ equals $C_2$, where the end $e_2$ adjacent to $v_2$ satisfies the degenerated line condition $L_{10}$. Figure \ref{Figure_21} shows that $C_{2,10}$ is not fixed by its conditions, i.e. $\mult(C_{2,10})=0$. If we consider $C_{2,01}$ instead, its multiplicity is $1$ since it is the absolute value of the determinant the following matrix $M(C_{2,01})$ (see Definition \ref{def:ev_mult})
\begin{align*}
\begin{array}{cc cc  ccc c}
  && \multicolumn{2}{c}{\footnotesize \textrm{Base $p_5$}}  &l_1&l_2&l_3&\\
  &\ldelim({5}{0.5em} &1&0& 0&0&0 & \rdelim){5}{0.5em} \\
  &&0&1 &0&0&0&\\
   \footnotesize\textrm{$L_{01}$}&&  0&1 & 1&0&0 & \\
\footnotesize\textrm{$L_2$}&&1&0 &0&-1&0&\\
\footnotesize\textrm{$L_3$}&&1&0&0&0&1&\\
\end{array}
\end{align*}
where $p_5$ is chosen as base point and the third row is associated to $L_{01}$ satisfied by $e_2$.

\end{example}

\begin{proposition}\label{prop:multiplicities_of_1/1_and_2/0}
Let $C$ be a tropical stable map that contributes to $N_{d}\left(p_{\underline{n}},L_{\underline{\kappa}}, \lambda_{[l-1]},\lambda'_l \right)$ such that $C$ has a contracted bounded edge $e$. The components arising from cutting $e$ as in Construction \ref{constr:cutting_contracted_bounded_edge} are denoted by $C_1,C_2$.
\begin{itemize}
\item[(a)] If $e$ is a $2/0$ edge, then\begin{align*}
\mult(C)=\mult(C_1)\cdot \mult(C_2).
\end{align*}
\item[(b)] If $e$ is a $1/1$ edge, then\begin{align*}
\mult(C)=|\mult(C_{1,10})\cdot \mult(C_{2,01})-\mult(C_{1,01})\cdot \mult(C_{2,10})|,
\end{align*}
where $C_{i,st}$ is defined in Notation \ref{notation:replacing_1/1-edge_conditions}.
\end{itemize}
\end{proposition}

\begin{proof}
It is sufficient to prove (a), (b) for ev-multiplicities only since the cross-ratio multiplicities can be expressed locally at vertices (see Proposition \ref{prop:cr_mult}). Thus contributions from vertices to cross-ratio multiplicities do not depend on cutting edges.
\begin{itemize}
\item[(a)]
Denote the vertices adjacent to $e$ by $v_1,v_2$ such that $v_1\in C_1$ and $v_2\in C_2$ and assume without loss of generality that $Y_1$ (notation from Definition \ref{def:1/1_and_2/0_edges}) is $0$-dimensional. Consider the ev-matrix $M(C)$ of $C$ of Definition \ref{def:ev_mult} with base point $v_1$, i.e.
\begin{align*}
M(C)=
\begin{array}{cc ccc|cccc !{\color{red!70!black}\vline width 1.5pt} ccccc}
  && \multicolumn{3}{c}{\footnotesize \textrm{Base $v_1$}} & \multicolumn{4}{c}{\footnotesize \textrm{lengths in $C_1$}}&\multicolumn{4}{c}{\footnotesize \textrm{lengths in $C_2$}} &\\
\footnotesize\textrm{conditions in $C_1$} &\ldelim({6}{0.5em}& \multicolumn{3}{c|}{\multirow{3}*{*}} & \multicolumn{4}{c!{\color{red!70!black}\vline width 1.5pt}}{\multirow{3}*{*}} & \multicolumn{4}{c}{\multirow{3}*{0}}
  & \rdelim){6}{0.5em} \\
  & &&& &&&& &&&& &\\
  & &&& &&&& &&&& &\\
  \arrayrulecolor{red!70!black}\Cline{1.5pt}{2-14}\arrayrulecolor{black} \footnotesize\textrm{conditions in $C_2$} && \multicolumn{3}{c|}{\multirow{3}*{*}} & \multicolumn{4}{c!{\color{red!70!black}\vline width 1.5pt}}{\multirow{3}*{0}} & \multicolumn{4}{c}{\multirow{3}*{*}} 
  &  \\
  & &&& &&&& &&&& &\\
  & &&& &&&& &&&& &\\
\end{array}
\end{align*}
Let $y_1$ be the number of rows that belong to the conditions in $C_1$, let $x_1$ be the number of columns belonging to the base point and the lengths in $C_1$.
Using notation from Definition \ref{def:1/1_and_2/0_edges}, we obtain
\begin{align*}
x_1&=2+3d_1-3+\#\underline{n_1}+\#\underline{\kappa_1}-\#\underline{l_1}+\#\underline{f_1}+1,\\
y_1&=2\cdot \#\underline{n_1}+\#\underline{\kappa_1}.
\end{align*}
On the other hand, $C_1$ is fixed by its set of conditions since $Y_1$ is $0$-dimensional, i.e. we can apply \eqref{eq:gen_pos_condition_count} for $m=\#\underline{n_1}+\#\underline{\kappa_1}+(\#\underline{f_1}+1)$ to obtain $x_1=y_1$.
Thus the bold red lines in $M(C)$ above divide $M(C)$ into squares, hence 
\begin{align*}
|\det(M(C))|=\mult(C_1)\cdot |\det(M)|,
\end{align*}
where $M$ is the square matrix on the bottom right. We define the matrix
\begin{align*}
M(C_{2,v_2}):=
\begin{array}{c cc !{\color{red!70!black}\vline width 1.5pt} ccc c}
  & \multicolumn{2}{c}{\footnotesize \textrm{Base $v_2$}}  &&&&\\
  \ldelim({5}{0.5em} &1&0& \multicolumn{3}{c}{\multirow{2}*{0}} & \rdelim){5}{0.5em} \\
  &0&1 &&&&\\
  \arrayrulecolor{red!70!black}\Cline{1.5pt}{1-7}\arrayrulecolor{black} &  \multicolumn{2}{c!{\color{red!70!black}\vline width 1.5pt}}{\multirow{3}*{*}} & \multicolumn{3}{c}{\multirow{3}*{$M$}} & \\
&& &&&\\
&&&&\\
\end{array}
\end{align*}
where the first two columns are chosen in such a way that $M(C_{2,v_2})$ is the ev-matrix of $C_2$ with respect to the base point $v_2$. Notice that
\begin{align*}
|\det(M)|=|\det(M(C_{2,v_2}))|
\end{align*}
and
\begin{align*}
|\det(M(C_{2,v_2}))|=\mult(C_2)
\end{align*}
hold, where $C_2$ satisfies the additional point condition imposed on $e_2$ by $Y_1$.
\item[(b)]
We assume that the weights of each multi line $\omega(L_k)$ (see Definition \ref{def:tropical_line}) for $k\in\kappa$ equals $1$ since we can pull out the factor $\omega(L_k)$ frome each row of the $ev$-matrix, apply all the following arguments and multiply with $\omega(L_k)$ later.

Denote the vertex of $C_1$ adjacent to the cut edge $e$ by $v_1$ and the other vertex adjacent to $e$ by $v_2$. The ev-matrix $M(C)$ of $C$ with respect to the base point $v_1$ is given by
\begin{align*}
M(C)=
\begin{array}{cc cc|cccc !{\color{red!70!black}\vline width 1.5pt} c| cccc c}
  && \multicolumn{2}{c}{\footnotesize \textrm{Base $v_1$}}  & \multicolumn{5}{c}{\footnotesize \textrm{lengths in $C_1$}}&\multicolumn{4}{c}{\footnotesize \textrm{lengths in $C_2$}}& \\
\footnotesize\textrm{conditions in $C_1$} &\ldelim({6}{0.5em}& \multicolumn{2}{c|}{\multirow{3}*{*}}  & \multicolumn{4}{c!{\color{red!70!black}\vline width 1.5pt}}{\multirow{3}*{*}}&* & \multicolumn{4}{c}{\multirow{3}*{0}} & \rdelim){6}{0.5em} \\
  & && &&&& &\vdots&&&& &\\
  & && &&&& &*&&&& &\\
  \arrayrulecolor{red!70!black}\Cline{1.5pt}{2-14}\arrayrulecolor{black} \footnotesize\textrm{conditions in $C_2$} && \multicolumn{2}{c|}{\multirow{3}*{*}} & \multicolumn{4}{c!{\color{red!70!black}\vline width 1.5pt}}{\multirow{3}*{0}}& 0 & \multicolumn{4}{c}{\multirow{3}*{*}} &  \\
  & && &&&& &\vdots&&&& &\\
  & && &&&& &0&&&& &\\
\end{array}
\end{align*}
The bold red lines divide $M(C)$ into square pieces at the upper left and the lower right. This follows from similar arguments used in the proof of part (a). Let $M$ be the matrix consisting of the lower right block of $M(C)$ whose entries (see above) are indicated by $*$ and its columns are associated to lengths in $C_2$.
Let $A=(a_{ij})_{ij}$ be the submatrix of $M(C)$ given by the rows that belong to conditions of $C_1$ and by the base point's columns and the columns that are associated to lengths in $C_1$, i.e. $A$ consists of all the $*$-entries above the bold red line in $M(C)$.

Consider the Laplace expansion of the rightmost column of $A$. Recursively, use Laplace expansion on every column that belongs to the lengths in $C_1$ starting with the rightmost column.
Eventually, we end up with a sum in which each summand contains a factor $\det (M_{a_{r1}a_{r2}})$ for a matrix $M_{a_{r1}a_{r2}}$, which is one of the following three matrices, namely
\begin{align*}
M_{a_{r1}a_{r2}}:=
\begin{array}{c cc|ccc c}
  & \multicolumn{2}{c}{} &\multicolumn{3}{c}{\footnotesize \textrm{lenghts in $C_2$}}&\\
  \ldelim({6}{0.5em} &a_{r1}&a_{r2}&0&\dots&0& \rdelim){6}{0.5em}\\
   \cline{1-7}  &\multicolumn{2}{c|}{\multirow{5}*{$*$}}&\multicolumn{3}{c}{\multirow{5}*{$M$}} & \\
  &&&&&&\\
  &&&&&&\\
  &&&&&&\\
  &&&&&&\\
\end{array},
\end{align*}
where $(a_{r1},a_{r2})=(1,0)$, $(a_{r1},a_{r2})=(0,1)$ or $(a_{r1},a_{r2})=(1,-1)$ are the remaining entries of $A$ in its $r$-th row after the recursive procedure.
Notice that in each of the three cases the entries of the first two columns are of such a form that $M_{st}$ for $st=10,01,1\text{-}1$ is the ev-matrix of $C_{2,st}$ (see Notation \ref{notation:replacing_1/1-edge_conditions}) with base point $v_2$.
We can group the summands according to the values $a_{r1},a_{r2}$ and obtain in total
\begin{align}\label{eq:proof5_multiplicities_of_1/1_and_2/0}
|\det(M(C))|=|F_{10}\cdot\det(M_{10})+F_{01}\cdot\det(M_{01})+F_{1\text{-}1}\cdot\det(M_{1\text{-}1})|,
\end{align}
where $F_{st}\in\mathbb{R}$ for $st=10,01,1\text{-}1$ are factors occuring due to the recursive Laplace expansion. More precisely, let $b$ be the number of bounded edges in $C_1$, i.e. the number of Laplace expansions we applied. Then
\begin{align}\label{eq:proof_multiplicities_of_1/1_and_2/0_F_st}
F_{st}=\sum_{r:(a_{r1},a_{r2})=(s,t)}\sum_{\sigma} \sgn(\sigma) \prod_{j=3}^{3+b} a_{\sigma(j)j},
\end{align}
where the second sum goes over all bijections $\sigma:\lbrace 3,\dots,3+b \rbrace\to \lbrace 1,\dots,r-1,r+1,\dots,b+1\rbrace$, i.e. it goes over all possibilities of choosing for each column Laplace expansion was used on an entry in a row of $A$ which is not the $r$-th row.

Let $A_{10},A_{01},A_{1\text{-}1}$ be the square matrices obtained from $A$ by adding the new first row $(1,0,0,\dots,0)$, $(0,1,0\dots,0)$ or $(1,-1,0,\dots,0)$ to $A$. Again, notice that $A_{st}$ for $st=10,01,1\text{-}1$ is the ev-matrix of $C_{1,st}$ (see Notation \ref{notation:replacing_1/1-edge_conditions}, Definition \ref{def:ev_mult}) with base point $v_1$.
We claim that
\begin{align}\label{eq:proof1_multiplicities_of_1/1_and_2/0}
\det(A_{10})=F_{01}-F_{1\text{-}1}
\end{align}
holds. Let $N$ be the number of columns and rows of $A_{st}$. Denote the entries of the $\ev$-matrix $M(C)$ by $m(C)_{ij}$. Define
\begin{align*}
S_{st}:=\lbrace r\in [N-1] \mid m(C)_{r1}=s,\; m(C)_{r2}=t \rbrace
\end{align*}
for $(s,t)=(1,0),(0,1),(1,-1)$ and notice that $\#S_{10}+\#S_{01}+\#S_{1\textrm{-}1}=N-1$. Denote the entries of $A_{10}$ by $a^{(10)}_{ij}$ and apply Leibniz' determinant formula to obtain
\begin{align*}
\det(A_{10})&=\sum_{\sigma\in\mathds{S}_N}\sgn(\sigma)\prod_{j=1}^N a^{(10)}_{\sigma(j)j}\\
&=\sum_{\substack{\sigma\in\mathds{S}_N\\ \sigma(2)\in S_{01}}}\sgn(\sigma)\prod_{j=1}^N a^{(10)}_{\sigma(j)j}
+
\sum_{\substack{\sigma\in\mathds{S}_N\\ \sigma(2)\in S_{1\textrm{-1}}}}\sgn(\sigma)\prod_{j=1}^N a^{(10)}_{\sigma(j)j}
=F_{01}-F_{1\text{-}1},
\end{align*}
where the second equality holds by definition of $S_{st}$ and the third equality holds by considering how contributions of $F_{01}$ and $F_{1\text{-}1}$ arise as choices of entries of $A$, see \eqref{eq:proof_multiplicities_of_1/1_and_2/0_F_st}. The minus sign comes from the factor $a^{(10)}_{\sigma(2),2}=-1$ in each product in the last sum. Thus \eqref{eq:proof1_multiplicities_of_1/1_and_2/0} holds.

We can show in a similar way that
\begin{align}
\det(A_{01})&=-\left(F_{10}+F_{1\text{-}1}\right)=-F_{10}-F_{1\text{-}1},\label{eq:proof2_multiplicities_of_1/1_and_2/0} \\
\det(A_{1\text{-}1})&=F_{10}+F_{1\text{-}1}+F_{01}-F_{1\text{-}1}=F_{10}+F_{01}\label{eq:proof3_multiplicities_of_1/1_and_2/0}
\end{align}
hold. Solving the system of linear equations \eqref{eq:proof1_multiplicities_of_1/1_and_2/0}, \eqref{eq:proof2_multiplicities_of_1/1_and_2/0}, \eqref{eq:proof3_multiplicities_of_1/1_and_2/0} for $F_{10},F_{01},F_{1\text{-}1}$ yields
\begin{align}\label{eq:proof4_multiplicities_of_1/1_and_2/0}
\left(\begin{array}{c}
F_{10}\\
F_{01}\\
F_{1\text{-}1}
\end{array}\right)\in
\left(\begin{array}{c}
-\det(A_{01})\\
\det(A_{10})\\
0
\end{array}\right)+
\langle\left(\begin{array}{c}
-1\\
1\\
1
\end{array}\right)\rangle,
\end{align}
where the $1$-dimensional part appears because of the relation
\begin{align*}
-\det(M_{10})+\det(M_{01})+\det(M_{1\text{-}1})=0.
\end{align*}
Combining \eqref{eq:proof5_multiplicities_of_1/1_and_2/0} with \eqref{eq:proof4_multiplicities_of_1/1_and_2/0} proves part (b), where $A_{st}=C_{1,st}$ and $M_{st}=C_{2,st}$.
\end{itemize}
\end{proof}

\section{General Kontsevich's formula}
In this section, we prove a general tropical Kontsevich's formula. For that, we must first deal with the behavior of the multiplicity of tropical stable maps under a split. More precisely, we would like to see that one summand in part (b) of Proposition \ref{prop:multiplicities_of_1/1_and_2/0} always vanishes.

\begin{definition}\label{def:split_respecting_CR}
Given a split $(d_1,\underline{n_1},\underline{\kappa_1},\underline{l_1},\underline{f_1}\mid d_2,\underline{n_2},\underline{\kappa_2},\underline{l_2},\underline{f_2})$ and a cross-ratio $\lambda'_l=(\beta_1 \beta_2 | \beta_3 \beta_4)$ with entries in $\underline{n_1}\cup\underline{\kappa_1}\cup\underline{f_1} \cup\underline{n_2}\cup\underline{\kappa_2}\cup\underline{f_2}$ and $\beta_1=\min_{i=1}^4(\beta_i)$ (the labels of ends of abstract tropical curves are natural numbers), we say that $(d_1,\underline{n_1},\underline{\kappa_1},\underline{l_1},\underline{f_1}\mid d_2,\underline{n_2},\underline{\kappa_2},\underline{l_2},\underline{f_2})$ is a split \textit{respecting} $\lambda'_l$ if $\beta_1,\beta_2\in\underline{n_1}\cup\underline{\kappa_1}\cup\underline{f_1}$ and $\beta_3,\beta_4\in\underline{n_2}\cup\underline{\kappa_2}\cup\underline{f_2}$. Using the minimum here prevents a factor of $\frac{1}{2}$ later, which would come from renaming $C_1$ to $C_2$ and vice versa.
\end{definition}

\begin{lemma}\label{lemma:counting_after_2/0_split}
Let $(d_1,\underline{n_1},\underline{\kappa_1},\underline{l_1},\underline{f_1}\mid d_2,\underline{n_2},\underline{\kappa_2},\underline{l_2},\underline{f_2})$ be a $2/0$ split of given general positioned conditions as in Remark \ref{remark:split_of_conditions} and Definition \ref{def:splits} that respects $\lambda'_l$ such that additionally $3d_1=|\underline{n_1}|+|\underline{l_1}|-|\underline{f_1}|$ holds. Then
\begin{align}\label{eq:lemma:counting_after_2/0_split}
\sum_{C:\; (d_1,\underline{n_1},\underline{\kappa_1},\underline{l_1},\underline{f_1}\mid d_2,\underline{n_2},\underline{\kappa_2},\underline{l_2},\underline{f_2})} \mult(C)
=
N_{d_1}\left(p_{\underline{n_1}},L_{\underline{\kappa_1}},\lambda^{\to e_1}_{\underline{l_1}} \right)
\cdot
N_{d_2}\left(p_{\underline{n_2}},p_{e_2},L_{\underline{\kappa_2}}, \lambda^{\to e_2}_{\underline{l_2}} \right)
\end{align}
holds, where the sum goes over all tropical stable maps $C$ with a contracted bounded edge $e$ such that $C$ contributes to $N_{d}\left(p_{\underline{n}},L_{\underline{\kappa}}, \lambda_{[l-1]},\lambda'_l \right)$, where $\lambda'_l$ is the large non-degenerated cross-ratio $C$ satisfies such that $C$ has a contracted bounded edge, and $C$ is of splitting type $(d_1,\underline{n_1},\underline{\kappa_1},\underline{l_1},\underline{f_1}\mid d_2,\underline{n_2},\underline{\kappa_2},\underline{l_2},\underline{f_2})$, and $p_{e_2}$ is a point condition imposed on $e_2$.
\end{lemma}

\begin{proof}
Each tropical stable map $C$ on the left-hand side of \eqref{eq:lemma:counting_after_2/0_split} can be cut at its contracted bounded edge as in Construction \ref{constr:cutting_contracted_bounded_edge} to obtain a tropical stable map $C_1$ that contributes to $N_{d_1}\left(p_{\underline{n_1}},L_{\underline{\kappa_1}},\lambda^{\to e_1}_{\underline{l_1}} \right)$ and a tropical stable map $C_2$ that contributes to $N_{d_2}\left(p_{\underline{n_2}},p_{e_2},L_{\underline{\kappa_2}}, \lambda^{\to e_2}_{\underline{l_2}} \right)$.

The other way around, each pair of tropical stable maps $C_1,C_2$ such that $C_1$ contributes to $N_{d_1}\left(p_{\underline{n_1}},L_{\underline{\kappa_1}},\lambda^{\to e_1}_{\underline{l_1}} \right)$ and $C_2$ contributes to $N_{d_2}\left(p_{\underline{n_2}},p_{e_2},L_{\underline{\kappa_2}}, \lambda^{\to e_2}_{\underline{l_2}} \right)$ can be glued to a tropical stable map $C$ using Remark \ref{remark:split_of_conditions}.

Proposition \ref{prop:multiplicities_of_1/1_and_2/0} states that
\begin{align*}
\mult(C)=\mult(C_1)\cdot\mult(C_2)
\end{align*}
and thus proves the lemma.
\end{proof}

\begin{lemma}\label{lemma:counting_after_1/1_split}
Let $(d_1,\underline{n_1},\underline{\kappa_1},\underline{l_1},\underline{f_1}\mid d_2,\underline{n_2},\underline{\kappa_2},\underline{l_2},\underline{f_2})$ be a $1/1$ split of given general positioned conditions as in Remark \ref{remark:split_of_conditions} and Definition \ref{def:splits} that respects $\lambda'_l$. Then
\begin{align}\label{eq:lemma:counting_after_1/1_split}
\sum_{C:\; (d_1,\underline{n_1},\underline{\kappa_1},\underline{l_1},\underline{f_1}\mid d_2,\underline{n_2},\underline{\kappa_2},\underline{l_2},\underline{f_2})} \mult(C)
=
N_{d_1}\left(p_{\underline{n_1}},L_{\underline{\kappa_1}}, L_{e_1} , \lambda^{\to e_1}_{\underline{l_1}} \right)
\cdot
N_{d_2}\left(p_{\underline{n_2}},L_{\underline{\kappa_2}}, L_{e_2} , \lambda^{\to e_2}_{\underline{l_2}} \right)
\end{align}
holds, where the sum goes over all tropical stable maps $C$ with a contracted bounded edge $e$ such that $C$ contributes to $N_{d}\left(p_{\underline{n}},L_{\underline{\kappa}}, \lambda_{[l-1]},\lambda'_l \right)$, where $\lambda'_l$ is the large non-degenerated cross-ratio $C$ satisfies such that $C$ has a contracted bounded edge, and $C$ is of splitting type $(d_1,\underline{n_1},\underline{\kappa_1},\underline{l_1},\underline{f_1}\mid d_2,\underline{n_2},\underline{\kappa_2},\underline{l_2},\underline{f_2})$, and $L_{e_i}$ for $i=1,2$ is a tropical multi line condition with ends of weight one that is imposed on $e_i$.
\end{lemma}

\begin{proof}
The ends of $Y_1$ and $Y_2$ (see Definition \ref{def:1/1_and_2/0_edges}) are of standard directions, i.e. of direction $(1,1),(-1,0)$ and $(0,-1)$ by Corollary \ref{corollary:1/1-edge_cut_standard_direction}. The position of $Y_1$ and $Y_2$ in $\mathbb{R}^2$ depends only on the position of the given conditions. In particular, moving the given conditions (while keeping the property of being in general position) moves $Y_1$ and $Y_2$ as well.

Assume that the given conditions are positioned in such a way that $Y_1$ and $Y_2$ intersect only in their ends as shown in Figure \ref{Figure_19}. Choose the multi line conditions $L_{e_1}$ and $L_{e_2}$ with weights one as in Figure \ref{Figure_19} and consider a tropical stable map $C_1$ that contributes to $N_{d_1}\left(p_{\underline{n_1}},L_{\underline{\kappa_1}}, L_{e_1} , \lambda^{\to e_1}_{\underline{l_1}} \right)$ and a tropical stable map $C_2$ that contributes to $N_{d_2}\left(p_{\underline{n_2}},L_{\underline{\kappa_2}}, L_{e_2} , \lambda^{\to e_2}_{\underline{l_2}} \right)$. The contracted end of $C_i$ for $i=1,2$ that satisfies $L_{e_i}$ is $e_i$. Let $v_i$ denote the vertex adjacent to $e_i$ for $i=1,2$.
Notice that $\ev_{e_i}(C_i)\in Y_i$, i.e. $C_i$ satisfies $Y_i$ by definition. Hence $v_i$ is a point in $Y_i\cdot L_{e_i}$ for $i=1,2$. Each pair of points $(v_1,v_2)$ is uniquely associated to a point $p$ in $Y_1\cdot Y_2$, see Figure \ref{Figure_19}. By Corollary \ref{corollary:1/1-edge_cut_v_i_3_valent_adjacent_to_end} each of the vertices $v_i$ is $3$-valent and adjacent to an end of $C_i$ for $i=1,2$. Hence (by moving $v_1,v_2$ along those ends) each pair of tropical stable maps $(C_1,C_2)$ as above can be glued to a tropical stable map $C$ as in Remark \ref{remark:split_of_conditions} such that the ends $e_1,e_2$ are glued to form a bounded edge that is contracted to $p$. On the other hand each tropical stable map $C$ on the left hand side of \eqref{eq:lemma:counting_after_1/1_split} can be split into a pair $(C_1,C_2)$ of tropical stable maps as above using Construction \ref{constr:cutting_contracted_bounded_edge}. Moreover,
\begin{align*}
\mult(C)=\mult(C_1)\cdot\mult(C_2)
\end{align*}
holds by Proposition \ref{prop:multiplicities_of_1/1_and_2/0} since $\mult(C_{1,01})$ and $\mult(C_{2,10})$ both vanish by our choice of positions of $Y_1$ and $Y_2$. Therefore \eqref{eq:lemma:counting_after_1/1_split} follows.

\begin{figure}[H]
\centering
\def\svgwidth{300pt}
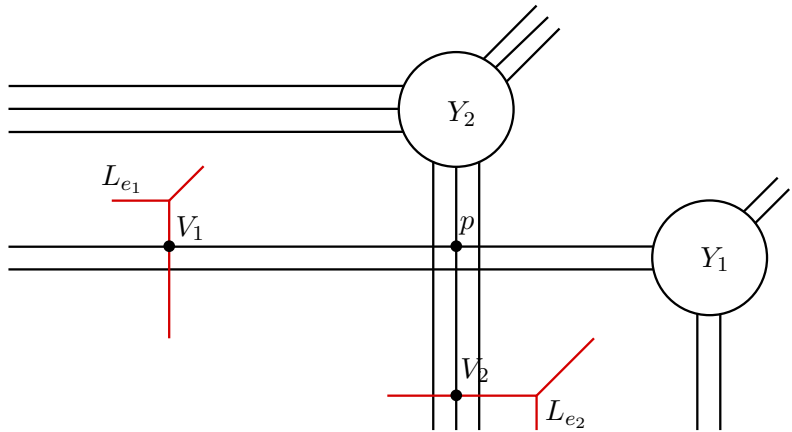
\caption{The $1$-dimensional conditions $Y_1$ and $Y_2$ after movement, together with the (multi) line conditions $L_{e_1}$ and $L_{e_2}$, where $p\in Y_1\cdot Y_2$ is the point associated to $V_1\in Y_1\cdot L_{e_1}$ and $V_2\in Y_2\cdot L_{e_2}$.}
\label{Figure_19}
\end{figure}

To finish the proof, we need to see that we can always assume that $Y_1$ and $Y_2$ intersect as shown in Figure \ref{Figure_19}, i.e. we want to show that the left hand side of \eqref{eq:lemma:counting_after_1/1_split} does not depend on the position of $Y_1$ and $Y_2$. Let $C$ be a tropical stable map contributing to $N_{d_1+d_2}\left(p_{\underline{n_1}},p_{\underline{n_2}},L_{\underline{\kappa_1}},L_{\underline{\kappa_2}},\lambda_{\underline{l_1}},\lambda_{\underline{l_2}},\lambda'_l \right)$ as in Proposition \ref{prop:contracted_bounded_edge}. Notice that $n\geq 1$ since we have a $1/1$ edge by Proposition \ref{prop:contracted_bounded_edge_no_point_conditions}. The cross-ratio's length $|\lambda_l'|$ is so large such that there is a contracted bounded edge $e$ in $C$, and $C$ is of splitting type $(d_1,\underline{n_1},\underline{\kappa_1},\underline{l_1},\underline{f_1}\mid d_2,\underline{n_2},\underline{\kappa_2},\underline{l_2},\underline{f_2})$.
Consider the cycle $Z_i$ that arises from forgetting the point conditions $p_{\underline{n_i}}$ and the line conditions $L_{\underline{\kappa_i}}$ for $i=1,2$ imposed on $C$. Hence $C$ gives rise to a top-dimensional cell of $Z_i$, where points in that cell correspond to $C$ together with some movement of the conditions $p_{\underline{n_i}},L_{\underline{\kappa_i}}$.
The proof of Proposition \ref{prop:contracted_bounded_edge} implies that if $|\lambda'_l|$ is large enough, then the given conditions can be moved in a bounded area $B$ (say $B\subset \mathbb{R}^2$ is a rectangular box) and all tropical stable maps that satisfy this moved conditions still have a contracted bounded edge.
Moreover, the splitting type of those tropical stable maps cannot change since that would require two contracted bounded edges which would contradict that our given conditions are in general position. Since $Z_1,Z_2$ are balanced, we might choose different positions for our point and line conditions for every splitting type without effecting the overall count.
Let $B_1,B_2\subset B$ be disjoint small rectangular boxes such that $B_1$ lies in the lower right corner of $B$ and $B_2$ lies in the upper left corner of $B$. Move the conditions $p_{\underline{n_1}} ,L_{\underline{\kappa_1}} ,\lambda_{\underline{l_1}}$ into $B_1$ and the conditions $p_{\underline{n_2}} ,L_{\underline{\kappa_2}} ,\lambda_{\underline{l_2}}$ into $B_2$ while maintaining their property of being in general position. By choosing $B_1$ and $B_2$ small enough, we can bring $Y_1$ and $Y_2$ in the desired position from Figure \ref{Figure_19}.
\end{proof}

\begin{theorem}[General Kontsevich's formula]\label{thm:generalized_kontsevich}
We use notation from Notation \ref{notation:underlined_symbols}, Definition \ref{def:1/1_and_2/0_edges}, \ref{def:split_respecting_CR} and Remark \ref{remark:split_of_conditions}.
Fix a degree $d$, point conditions $p_{\underline{n}}$, multi line conditions $L_{\underline{\kappa}}$ and degenerated cross-ratios $\lambda_{[l]}$ such that these conditions are in general position.
Let $\lambda'_l$ denote a cross-ratio that degenerates to $\lambda_l$. 
\begin{itemize}
\item[(a)]
If there is at least one point condition, i.e. $p_{\underline{n}}\neq \emptyset$, then the equation
\begin{align}\label{eq:generalized_kontsevich}
\begin{split}
N_{d}\left(p_{\underline{n}},L_{\underline{\kappa}}, \lambda_{[l]} \right)
=
&\sum_{\substack{(d_1,\underline{n_1},\underline{\kappa_1},\underline{l_1},\underline{f_1}\mid d_2,\underline{n_2},\underline{\kappa_2},\underline{l_2},\underline{f_2})\\ \textrm{ is a $1/1$ split respecting }\lambda'_l}}
N_{d_1}\left(p_{\underline{n_1}},L_{\underline{\kappa_1}}, L_{e_1} , \lambda^{\to e_1}_{\underline{l_1}} \right)
\cdot
N_{d_2}\left(p_{\underline{n_2}},L_{\underline{\kappa_2}}, L_{e_2} , \lambda^{\to e_2}_{\underline{l_2}} \right)\\
&+\sum_{\substack{(d_1,\underline{n_1},\underline{\kappa_1},\underline{l_1},\underline{f_1}\mid d_2,\underline{n_2},\underline{\kappa_2},\underline{l_2},\underline{f_2})\\ \textrm{ is a $2/0$ split respecting $\lambda'_l$ and } \\ 3d_1=|\underline{n_1}|+|\underline{l_1}|-|\underline{f_1}| }}
N_{d_1}\left(p_{\underline{n_1}},L_{\underline{\kappa_1}},\lambda^{\to e_1}_{\underline{l_1}} \right)
\cdot
N_{d_2}\left(p_{\underline{n_2}},p_{e_2},L_{\underline{\kappa_2}}, \lambda^{\to e_2}_{\underline{l_2}} \right)\\
&+\sum_{\substack{(d_1,\underline{n_1},\underline{\kappa_1},\underline{l_1},\underline{f_1}\mid d_2,\underline{n_2},\underline{\kappa_2},\underline{l_2},\underline{f_2})\\ \textrm{ is a $2/0$ split respecting $\lambda'_l$ and } \\ 3d_2=|\underline{n_2}|+|\underline{l_2}|-|\underline{f_2}| }}
N_{d_1}\left(p_{\underline{n_1}},p_{e_1},L_{\underline{\kappa_1}},\lambda^{\to e_1}_{\underline{l_1}} \right)
\cdot
N_{d_2}\left(p_{\underline{n_2}},L_{\underline{\kappa_2}}, \lambda^{\to e_2}_{\underline{l_2}} \right)
\end{split}
\end{align}
holds.
\item[(b)]
If there are no point conditions, i.e. $p_{\underline{n}}=\emptyset$, then the equation
\begin{align}\label{eq:initial_values_recursion}
N_{d}\left(L_{\underline{\kappa}}, \lambda_{[l]} \right)
=
\sum_{\substack{(\underline{l_1},\underline{f_1}\mid \underline{l_2},\underline{f_2})\\ \textrm{ is a $2/0$ split respecting }\lambda'_l}}
N_{0}\left(L_a,L_b, \lambda_{\underline{l_1}}^{\to e} \right) \cdot
N_{d}\left(p,L_{\underline{\kappa}}\backslash \lbrace L_a,L_b \rbrace, \lambda_{\underline{l_2}}^{\to e}\right)
\end{align}
holds, where the line conditions $L_a,L_b$ are the ones of Lemma \ref{lemma:initial_values_1}.
\end{itemize}
\ \\
Moreover, \eqref{eq:generalized_kontsevich} and \eqref{eq:initial_values_recursion} give rise to a recursion with two types of initial values:
\begin{itemize}
\item[(1)]
The numbers $N_{d}\left(p_{\underline{n}}\right)$ which tropical Kontsevich's formula (Corollary \ref{cor:tropical_kontsevich_formula}) provides.
\item[(2)]
The numbers $N_{0}\left(L_a,L_b, \lambda_{\underline{l_1}}^{\to e} \right)$ which satisfy
\begin{align}\label{eq:initial_values_CR_mult}
N_{0}\left(L_a,L_b, \lambda_{\underline{l_1}}^{\to e} \right)
=
\omega(L_a)\cdot\omega(L_b)\cdot \mult_{\CR}(v'),
\end{align}
where $v'$ denotes the only vertex of the only tropical stable map contributing to $N_{0}\left(L_a,L_b, \lambda_{\underline{l_1}}^{\to e} \right)$ and $\mult_{\CR}(v')$ is its cross-ratio multiplicity, see Definition \ref{def:CR_mult}. Notice that in the special case of $\lambda_{\underline{l_1}}^{\to e}=\emptyset$ we have
\begin{align}
N_{0}\left(L_a,L_b \right)
=
\omega(L_a)\cdot\omega(L_b).
\end{align}
\end{itemize}
\end{theorem}

Using Tyomkin's correspondence theorem \ref{thm:correspondence_thm_CRC} and Remark \ref{remark:numbers_independent_of_positions_and_lengths}, Theorem \ref{thm:generalized_kontsevich} immediately yields the following corollary.

\begin{corollary}[Non-tropical general Kontsevich's formula]\label{cor:non-tropical_general_Kontsevich_formula}
Let $N^{\operatorname{class}}_{d}\left(p_{\underline{n}}, \mu_{[l]} \right)$ denote the number of plane rational degree $d$ curves that satisfy point conditions and non-tropical cross-ratios $\mu_1,\dots,\mu_l$ as in Theorem \ref{thm:correspondence_thm_CRC} such that all conditions are in general position. Then Theorem \ref{thm:generalized_kontsevich} provides a recursive formula to calculate these numbers with initial values as in Theorem \ref{thm:generalized_kontsevich}.
\end{corollary}


\begin{example}
We want to give an example of how to compute numbers we are looking for using our general Kontsevich's formula. Say we want to compute $N_{2}\left(p_{[3]},L_{4},L_{5}, \lambda_{[2]} \right)$. For degenerated cross-ratios
\begin{align*}
\lambda_1:=\lbrace 1,2,3,4\rbrace \quad \textrm{ and } \quad\lambda_2:=\lbrace 1,2,3,5\rbrace.
\end{align*}
Notice that \eqref{eq:gen_pos_condition_count} is satisfied so your input data makes sense. 
Recall the conventions we used for labeling ends: in this example, we want to count tropical stable maps $C$ of degree $2$ in $\mathbb{R}^2$ that have $5$ contracted ends. A contracted end labeled with $i$ satisfies the point condition $p_i$ for $i=1,2,3$ and satisfies the multi line condition $L_i$ for $i=4,5$. There is no non-contracted end which satisfies no condition. To use Theorem \ref{thm:generalized_kontsevich}, we need to fix a cross-ratio $\lambda_2'$ that degenerates to $\lambda_2$. We choose
\begin{align*}
\lambda'_2:=(12|35).
\end{align*}
If $C$ splits into $C_1,C_2$, then by Definition \ref{def:split_respecting_CR} ends $1,2$ are contracted ends of $C_1$, i.e. $p_1,p_2$ are satisfied in $C_1$, and $3,5$ are contracted ends of $C_2$, i.e. $p_3,L_5$ are satisfied in $C_2$. Therefore $\lambda_1$ is satisfied in $C_1$ such that $4$ is a contracted end of $C_1$ that satisfies $L_4$. If we go through the three cases of different types of splits using \eqref{eq:equation_for_1/1_split} and \eqref{eq:equation_for_2/0_split}, we see that the only possible split is the $2/0$ split
\begin{align*}
(1,p_1,p_2,L_4,\lambda_1\mid 1,p_3,L_5).
\end{align*}
Hence part (a) of Theorem \ref{thm:generalized_kontsevich} yields
\begin{align*}
N_{2}\left(p_{[3]},L_{4},L_{5}, \lambda_{[2]} \right)
=
N_{1}\left(p_1,p_2,L_4, \lambda_{1}^{\to e_1} \right)
\cdot
N_{1}\left(p_3,p_{e_2},L_5 \right),
\end{align*}
where the rightmost factor can be written as
\begin{align*}
N_{1}\left(p_3,p_{e_2},L_5 \right)=\omega(L_5)
\cdot
\underbrace{N_{1}\left(p_3,p_{e_2}\right)}_{=1}
\end{align*}
by tropical B\'ezout's Theorem \cite{FirstStepsIntersectionTheory}.

So it remains to calculate $N_{1}\left(p_1,p_2,L_4, \lambda_{1}^{\to e_1} \right)$. For that, we want to use Theorem \ref{thm:generalized_kontsevich} again. A stable map $C$ contributing to $N_{1}\left(p_1,p_2,L_4, \lambda_{1}^{\to e_1} \right)$ has $4$ contracted ends. A contracted end labeled with $i$ satisfies $p_i$ for $i=1,2$ and $L_i$ for $i=4$. The remaining contracted end is labeled with $e_1$ and satisfies no point condition. To stick to our convention of labeling ends with natural numbers, we relabel $e_1$ by $6$. Again, we need to fix a cross-ratio $\lambda'^{\to e_1}_1$ that degenerates to $\lambda^{\to e_1}_1=\lbrace 1,2,6,4\rbrace$. We choose
\begin{align*}
\lambda'^{\to e_1}_1:=(12|46).
\end{align*}
If $C$ splits into $C_1,C_2$ then $1,2$ are contracted ends of $C_1$, i.e. $p_1,p_2$ are satisfied in $C_1$, and $4,6$ are contracted ends of $C_2$, i.e. $L_4$ is satisfied by $C_2$ and there is one contracted end, labeled $6$, in $C_2$ that satisfies no condition. As before, we can go through all splits and notice that
\begin{align*}
(1,p_1,p_2\mid 0,L_4,6)
\end{align*}
is the only possible split. Hence part (a) of Theorem \ref{thm:generalized_kontsevich} yields
\begin{align*}
N_{1}\left(p_1,p_2,L_4, \lambda_{1}^{\to e_1} \right)
=
N_{1}\left(p_1,p_2,L_{e_1'} \right)
\cdot
N_{0}\left(L_4,L_{e_2'} \right),
\end{align*}
where
\begin{align*}
N_{1}\left(p_1,p_2,L_{e_1'} \right)=\underbrace{\omega(L_{e_1'})}_{=1}\cdot \underbrace{N_{1}\left(p_1,p_2\right)}_{=1}
\end{align*}
by tropical B\'ezout's Theorem and by Defintion of $L_{e_1'}$, and $N_{0}\left(L_4,L_{e_2'} \right)=\omega(L_4)$ by Theorem \ref{thm:generalized_kontsevich}.

In total, we calculated
\begin{align*}
N_{2}\left(p_{[3]},L_{4},L_{5}, \lambda_{[2]} \right)=\omega(L_4)\cdot \omega(L_5)
\end{align*}
for $\lambda_1,\lambda_2$ defined as above.
\end{example}

We now prove Theorem \ref{thm:generalized_kontsevich}, discuss the initial values of the recursion Theorem \ref{thm:generalized_kontsevich} provides and then proceed with tropical Kontsevich's formula which is a corollary of part (a) of Theorem \ref{thm:generalized_kontsevich}.

\begin{proof}[Proof of part (a) of Theorem \ref{thm:generalized_kontsevich}]
Using Remark \ref{remark:numbers_independent_of_positions_and_lengths}, we obtain
\begin{align*}
N_{d}\left(p_{\underline{n}},L_{\underline{\kappa}}, \lambda_{[l]} \right)
=
N_{d}\left(p_{\underline{n}},L_{\underline{\kappa}}, \lambda_{[l-1]},\lambda'_l \right)
\end{align*}
for a cross-ratio $\lambda'_l$ that degenerates to $\lambda_l$. Since $N_{d}\left(p_{\underline{n}},L_{\underline{\kappa}}, \lambda_{[l-1]},\lambda'_l \right)$ does not depend on $|\lambda'_l|$, choose it to be large as in Proposition \ref{prop:contracted_bounded_edge}. Hence each stable map contributing to $N_{d}\left(p_{\underline{n}},L_{\underline{\kappa}}, \lambda_{[l-1]},\lambda'_l \right)$ has a contracted bounded edge $e$ which can be cut as using Construction \ref{constr:cutting_contracted_bounded_edge} and thus gives rise to some splitting type $(d_1,\underline{n_1},\underline{\kappa_1},\underline{l_1},\underline{f_1}\mid d_2,\underline{n_2},\underline{\kappa_2},\underline{l_2},\underline{f_2})$ that respects $\lambda'_l$. Therefore
\begin{align}\label{eq:proof_generalized_kontsevich_1}
N_{d}\left(p_{\underline{n}},L_{\underline{\kappa}}, \lambda_{[l-1]},\lambda'_l \right)
=
\sum_{\substack{(d_1,\underline{n_1},\underline{\kappa_1},\underline{l_1},\underline{f_1}\mid d_2,\underline{n_2},\underline{\kappa_2},\underline{l_2},\underline{f_2})\\ \textrm{ is a split respecting }\lambda'_l}}
\sum_{C} \mult(C),
\end{align}
where the second sum goes over all stable maps $C$ that give rise to the split $(d_1,\underline{n_1},\underline{\kappa_1},\underline{l_1},\underline{f_1}\mid d_2,\underline{n_2},\underline{\kappa_2},\underline{l_2},\underline{f_2})$. Reordering the first sum of \eqref{eq:proof_generalized_kontsevich_1} as in \eqref{eq:generalized_kontsevich} and applying Lemmas \ref{lemma:counting_after_2/0_split}, \ref{lemma:counting_after_1/1_split} proves part (a) of Theorem \ref{thm:generalized_kontsevich}.
\end{proof}

\begin{proof}[Proof of part (b) of Theorem \ref{thm:generalized_kontsevich}]
We use notation from Lemma \ref{lemma:initial_values_1}, \ref{lemma:initial_values_2} and Proposition \ref{prop:contracted_bounded_edge_no_point_conditions}.
We use Remark \ref{remark:numbers_independent_of_positions_and_lengths}, i.e.
\begin{align}\label{eq:initial_values_proof_1}
N_{d}\left(L_{\underline{\kappa}}, \lambda_{[l]} \right)
=
N_{d}\left(L_{\underline{\kappa}}, \lambda_{[l-1]},\lambda_l' \right),
\end{align}
and conclude with Proposition \ref{prop:contracted_bounded_edge_no_point_conditions} that each stable map contributing to the right hand side of \eqref{eq:initial_values_proof_1} has a contracted bounded edge $e$ which is adjacent to a vertex $v'$ which is in turn adjacent to $e_1,e_2$. Notice that cutting $e$ yields a $2/0$ split. Thus Lemma \ref{lemma:counting_after_2/0_split} gives us equation \eqref{eq:initial_values_CR_mult}.
\end{proof}

\begin{proof}[Proof of the initial values part of Theorem \ref{thm:generalized_kontsevich}]
Notice that equations \eqref{eq:generalized_kontsevich} and \eqref{eq:initial_values_recursion} of Theorem \ref{thm:generalized_kontsevich} allow us to successively reduce the number of point, multi line or cross-ratio conditions. There are three cases:
\begin{itemize}
\item[(1)]
We run out of cross-ratio conditions. Then, if there are point conditions left, tropical B\'ezout's Theorem \cite{FirstStepsIntersectionTheory} can be applied to reduce the initial value problem to the numbers $N_{d}\left(p_{\underline{n}}\right)$ which tropical Kontsevich's formula (Corollary \ref{cor:tropical_kontsevich_formula}) provides. If there are no point conditions left, then
\begin{align*}
N_{d}\left(L_{\underline{\kappa}}\right)=0
\end{align*}
for all $d\neq 0$ applies. Otherwise $d=0$, $\#\underline{\kappa}=\#\lbrace a,b\rbrace=2$ and $\#\underline{f}=1$ such that
\begin{align*}
N_{0}\left(L_{{\kappa}}\right)=\omega(L_a)\cdot\omega(L_b)
\end{align*}
holds.
\item[(2)]
We run out of point conditions. Then \eqref{eq:initial_values_recursion} reduces the initial value problem to calculating $N_{0}\left(L_a,L_b, \lambda_{\underline{l_1}}^{\to e} \right)$. This can be done via \eqref{eq:initial_values_CR_mult}.

For equation \eqref{eq:initial_values_CR_mult}, notice that each edge of a tropical stable map of degree $0$ must be contracted. Thus there cannot be a bounded edge since all cross-ratios are degenerated. Hence there is exactly one vertex $v'$ in such a stable map whose position is determined by the unique point of intersection of $L_a$ and $L_b$. Therefore there is exactly one stable map contributing to $N_{0}\left(L_a,L_b, \lambda_{\underline{l_1}}^{\to e} \right)$ whose multiplicity is $\omega(L_a)\cdot\omega(L_b)\cdot \mult_{\CR}(v')$ by Proposition \ref{prop:cr_mult}.
\item[(3)]
We run out of multi line conditions. Then \eqref{eq:generalized_kontsevich} can still be applied, so cases (1) and (2) apply.
\end{itemize}
\end{proof}

\begin{corollary}[Tropical Kontsevich's formula, \cite{KontsevichPaper}]\label{cor:tropical_kontsevich_formula}
For $\#\underline{n}=3d-1>0$ general positioned point conditions the equality
\begin{align*}
N_{d}\left( p_{\underline{n}} \right)=\sum_{\substack{d_1+d_2=d \\ d_1,d_2>0}} \left( d_1^2 d_2^2 \cdot {3d-4 \choose 3d_1-2}- d_1^ 3d_2 \cdot {3d-4 \choose 3d_1-1} \right) N_{d_1}\left( p_{\underline{n_1}} \right)N_{d_2}\left( p_{\underline{n_2}} \right)
\end{align*}
holds and provides a recursion to calculate $N_{d}\left( p_{\underline{n}} \right)$ from the initial value $N_{1}\left( p_1,p_2 \right)=1$.
\end{corollary}

\begin{proof}
Let $p_{\underline{n}}$ be point conditions, let $L_a,L_b$ be line conditions, i.e. multi lines with weights $\omega(L_a)=\omega(L_b)=1$ and let $\lambda=\lbrace L_a,L_b,p_c,p_d \rbrace $ be a degenerated cross-ratio, where $p_c,p_d\in p_{\underline{n}}$ are points and the labels are chosen in such a way that $a<b<c<d$.

Consider the cross-ratio $\lambda':=(L_ap_c|L_bp_d)$ that degenerates to $\lambda$.
We claim that \eqref{eq:generalized_kontsevich} reduces to
\begin{align}\label{eq:corollary_kontsevich_claim1}
\begin{split}
N_{d}\left(p_{\underline{n}},L_a,L_b, \lambda \right)
=
&\sum_{\substack{(d_1,\underline{n_1}\mid d_2,\underline{n_2})\\ \textrm{ is a $1/1$ split respecting }\lambda'}}
N_{d_1}\left(p_{\underline{n_1}}, L_a, L_{e_1} \right)
\cdot
N_{d_2}\left(p_{\underline{n_2}}, L_b, L_{e_2}  \right)
\end{split}
\end{align}
in our case. Since we only have two line conditions and no contracted ends without point or line conditions, each split we deal with can be written as $(d_1,\underline{n_1}\mid d_2,\underline{n_2})$ since $\lambda'$ determines the distribution of $L_a$ and $L_b$ in each possible split respecting $\lambda'$. To show the claim, it remains to show that the last two sums of \eqref{eq:generalized_kontsevich} vanish. For that it is, because of symmetry, sufficient to show that the second sum vanishes. Let $N_{d_1} \left(p_{\underline{n_1}},L_a\right) \cdot N_{d_2}\left(p_{\underline{n_2}},p_{e_2},L_b \right)$ be a factor of the second sum. Let $C_1$ be a tropical stable map contributing to $N_{d_1} \left(p_{\underline{n_1}},L_a\right)$ and let $C_2$ be a tropical stable map contributing to $N_{d_2}\left(p_{\underline{n_2}},p_{e_2},L_b \right)$. Using Remark \ref{remark:split_of_conditions}, $C_1$ and $C_2$ can be glued to form a tropical stable map $C$ which has a contracted bounded edge $e$. Since our split was a $2/0$ split, the $3$-valent vertex $v_1$ of $C$ that is adjacent to $e$ is fixed. Hence there is a contracted end satisfying a point condition that is adjacent to $v_1$. Thus there is another contracted end adjacent to $v_1$ which needs to satisfy either a point or a line condition which is a contradiction since all conditions are in general position.

Now consider the cross-ratio $\tilde{\lambda}':=(L_aL_b|p_cp_d)$ that degenerates to $\lambda$.
We claim that \eqref{eq:generalized_kontsevich} reduces to
\begin{align}\label{eq:corollary_kontsevich_claim2}
\begin{split}
N_{d}\left(p_{\underline{n}},L_a,L_b, \lambda \right)
=
&\sum_{\substack{(d_1,\underline{n_1}\mid d_2,\underline{n_2})\\ \textrm{ is a $1/1$ split respecting }\tilde{\lambda}'}}
N_{d_1}\left(p_{\underline{n_1}}, L_a,L_b, L_{e_1} \right)
\cdot
N_{d_2}\left(p_{\underline{n_2}}, L_{e_2}  \right)\\
&+
N_{0}\left( L_a,L_b \right)
\cdot
N_{d}\left(p_{\underline{n}}, p_{e_2}  \right)
\end{split}
\end{align}
in this case. As before, splits can be written as $(d_1,\underline{n_1}\mid d_2,\underline{n_2})$. The last sum of \eqref{eq:generalized_kontsevich} vanishes with the same arguments from before. It remains to see that the second sum of \eqref{eq:generalized_kontsevich} equals the product $N_{0}\left( L_a,L_b \right) \cdot N_{d}\left(p_{\underline{n}}, p_{e_2}  \right)$. If $d_1>0$ and we consider a product contributing to the last sum, then the same arguments from before show that this product vanishes. Hence the only remaining contribution from the second sum that is possible is $N_{0}\left( L_a,L_b \right) \cdot N_{d}\left(p_{\underline{n}}, p_{e_2}  \right)$.

Notice that there are no cross-ratios on the right-hand sides of \eqref{eq:corollary_kontsevich_claim1} and \eqref{eq:corollary_kontsevich_claim2} such that tropical B\'ezout's Theorem \cite{FirstStepsIntersectionTheory} yields
\begin{align*}
\begin{split}
&\sum_{\substack{(d_1,\underline{n_1}\mid d_2,\underline{n_2})\\ \textrm{ is a $1/1$ split respecting }\lambda'}}
d_1^2 N_{d_1}\left(p_{\underline{n_1}}\right)
\cdot
d_2^2 N_{d_2}\left(p_{\underline{n_2}}\right)\\
&=
\sum_{\substack{(d_1,\underline{n_1}\mid d_2,\underline{n_2})\\ \textrm{ is a $1/1$ split respecting }\tilde{\lambda}'}}
d_1^3 N_{d_1}\left(p_{\underline{n_1}} \right)
\cdot
d_2 N_{d_2}\left(p_{\underline{n_2}} \right)
+
N_{0}\left( L_a,L_b \right)
\cdot
N_{d}\left(p_{\underline{n}}, p_{e_2}  \right)
\end{split}
\end{align*}
since $\omega(L_a)=\omega(L_b)=\omega(L_{e_1})=\omega(L_{e_2})=1$. Using $N_{0}\left( L_a,L_b \right)=\omega(L_a)\omega(L_b)=1$, we obtain
\begin{align*}
N_{d}\left(p_{\underline{n}}, p_{e_2}  \right) =\\
\sum_{\substack{(d_1,\underline{n_1}\mid d_2,\underline{n_2})\\ \textrm{ is a $1/1$ split respecting }\lambda'}}
&d_1^2 d_2^2N_{d_1}\left(p_{\underline{n_1}}\right)
N_{d_2}\left(p_{\underline{n_2}}\right)
-
\sum_{\substack{(d_1,\underline{n_1}\mid d_2,\underline{n_2})\\ \textrm{ is a $1/1$ split respecting }\tilde{\lambda}'}}
d_1^3 d_2 N_{d_1}\left(p_{\underline{n_1}} \right)
N_{d_2}\left(p_{\underline{n_2}} \right).
\end{align*}
Since all conditions we started with are in general position
\begin{align*}
3d=\#\underline{n}+1+1
\end{align*}
holds, i.e. each choice of $\underline{n_1},\underline{n_2}$ in a split for fixed $d_1,d_2$ is a choice of distributing the remaining $3d-4$ points. There are ${3d-4 \choose 3d_1-2}$ choices if $p_c\in\underline{n_1}$ and ${3d-4 \choose 3d_1-1}$ choices if $p_c,p_d\in\underline{n_2}$. Using $3d_i=\#\underline{n_i}+1$ provides the index for the sum we are looking for.
\end{proof}

\subsection*{Further generalizations}
The same methods Gathmann and Markwig used to prove tropical Kontsevich's formula \cite{KontsevichPaper} also yield a recursive formula for counting rational tropical stable maps of bidegree $(d_1,d_2)$ (i.e. with ends of directions $(\pm 1,0),(0,\pm 1)$) to $\mathbb{R}^2$ that satisfy point conditions, see \cite{FranzMarkwig}. Analogously, the methods developed in this paper yield a recursive formula for rational tropical stable maps to $\mathbb{R}^2$ of bidegree $(d_1,d_2)$ that satisfy point conditions, \textit{degenerated} multi line conditions and cross-ratio conditions.

\hyphenation{Kaisers-lautern} 
\bibliographystyle{alpha}
\bibliography{literatur}

\end{document}